\documentclass[12pt,reqno]{amsart}
\usepackage{amsthm,amsmath,amsfonts, amssymb,url,cite,color}
\usepackage[pagebackref]{hyperref}
\hypersetup{colorlinks=true}
\usepackage[dvipsnames]{xcolor}

\usepackage{enumerate}
\usepackage{tikz}
\usepackage{tikzscale}
\usepackage{graphicx}
\usepackage{graphics}
\usepackage{pict2e}
\usepackage{collectbox}
\usepackage{cases}
\usepackage{hyperref} 
\numberwithin{equation}{section}
\usepackage{cite}
\newtheorem{thm}{Theorem}
\newtheorem{lem}{Lemma}
\newtheorem{prop}[thm]{Proposition}

\newtheorem{cor}[thm]{Corollary}
\newtheorem{defi}[thm]{Definition}
\newtheorem{conj}[thm]{Conjecture}

\newtheorem{exam}{Example}
\newtheorem{remark}{Remark}
\def\la{\lambda}

\def\cyc{\textrm{cyc}}

\def\lma{\mathrm{lma}}
\def\rmi{\mathrm{rmi}}
\def\remi{\mathrm{remi}}
\def\romi{\mathrm{romi}}
\def\lema{\mathrm{lema}}
\def\loma{\mathrm{loma}}
\def\romi{\mathrm{romi}}

\def\TB{\mathrm{dom}}
\def\dom{\mathrm{dom}}
\def\Orb{\textrm{Orb}}

\def\F{\textrm{first}}
\def\L{\textrm{last}}

\def\X{\mathcal{X}}
\def\Y{\mathcal{Y}}

\mathchardef\mhyphen="2D
\DeclareMathOperator\des{des}
\DeclareMathOperator\drop{drop}
\DeclareMathOperator\les{(31\mhyphen 2)}
\DeclareMathOperator\less{(13\mhyphen 2)}
\DeclareMathOperator\res{(2\mhyphen 13)}
\DeclareMathOperator\ress{(2\mhyphen 31)}
\def\312{\les}
\def\132{\less}
\def\213{\res}
\def\231{\ress}
\newcommand{\overbar}[1]{\mkern 1.5mu\overline{\mkern-1.5mu#1\mkern-1.5mu}\mkern 1.5mu}

\def\P{\mathcal{P}}

\def\dd{\rm {dd}}

\def\Z{\mathbb{Z}}
\def\N{\mathbb{N}}
\def\E{{\mathcal E}}
\def\S{{\mathfrak S}}

\def\blue{\textcolor{blue}}
\def\red{\textcolor{red}}

\newcommand{\bea}{\begin{align}} 
\newcommand{\ena}{\end{align}}

\newcommand\vartextvisiblespace[1][.5em]{%
  \makebox[#1]{%
    \kern.07em
    \vrule height.3ex
    \hrulefill
    \vrule height.3ex
    \kern.07em
  }
}
\newcommand{\x}{\vartextvisiblespace}
\def\a{\alpha}
\def\Z{\mathbb{Z}}
\def\N{\mathbb{N}}
\def\S{\mathfrak{S}}
\def\D{\mathfrak{D}}

\def\cyc{\mathop {\rm cyc}}
\def\f{\mathop {\rm fix}}

\def\mi{\mathop {\rm mo}}
\def\fd{\mathop {\rm fd}}
\def\snd{\mathop {\rm si}}
\def\mp{\mathop {\rm me}}
\def\fnd{\mathop {\rm fi}}
\def\sd{\mathop {\rm sd}}

\begin{document}

\title[Even-odd drop permutations]{
Cycles of even-odd drop permutations and  continued fractions  of
Genocchi numbers}\thanks{\today}
\author{Qiongqiong Pan and Jiang Zeng}
\address{College of Mathematics and Physics, Wenzhou University\\
Wenzhou 325035, PR China}
\email{qpan@math.univ-lyon1.fr}
\address{Universit\'e de Lyon,  Universit\'e Lyon 1,
UMR 5208 du CNRS, Institut Camille Jordan\\
F-69622, Villeurbanne Cedex, France}
\email{zeng@math.univ-lyon1.fr}
\begin{abstract} 
Recently, 
Lazar and Wachs (arXiv:1910.07651) showed that the  (median) Genocchi numbers play a fundamental role in the study of the homogenized Linial arrangement and obtained two new permutation models 
 (called D-permutations and E-permutations) for  (median) Genocchi numbers.
 They further  conjecture that the distributions of cycle numbers over the two models are equal. 
 In a follow-up,  
  Eu et al. (arXiv:2103.09130)  further proved the  gamma-positivity of 
 the descent polynomials of  
even-odd descent permutations, which are in bijection with 
 E-permutations by Foata's fundamental transformation.  This paper merges the above two papers by considering a general moment sequence which encompasses 
the number of 
 cycles and number of drops of E-permutations. 
  Using the  combinatorial theory 
of  continued fraction, the moment connection  enables us to confirm Lazar-Wachs' conjecture and  obtain   a natural 
$(p,q)$-analogue of Eu et al's descent polynomials. Furthermore,
 we show that  
the  $\gamma$-coefficients  of our $(p,q)$-analogue   of  descent polynomials have
 the same  factorization flavor as 
 the $\gamma$-coeffcients of Br\"and\'en's $(p,q)$-Eulerian polynomials. 
\end{abstract}

\maketitle 
 \vskip 0.5cm 
\tableofcontents
\section{Introduction}
The Genocchi numbers $G_{2n}$ and median Genocchi numbers $H_{2n+1}$ are well studied, and have seen recent attention and new combinatorial interpretations, see \cite{BN20, Du74, BD79, DR94,BJS10,  He19, LW20, Vi82} and \url{https://oeis.org/A005439}. 
These two allied sequences of numbers can be easily defined by    the  \emph{Seidel triangle}~\cite{Seidel1877}  through a
boustrophedon algorithm as follows:
\begin{figure}[h]
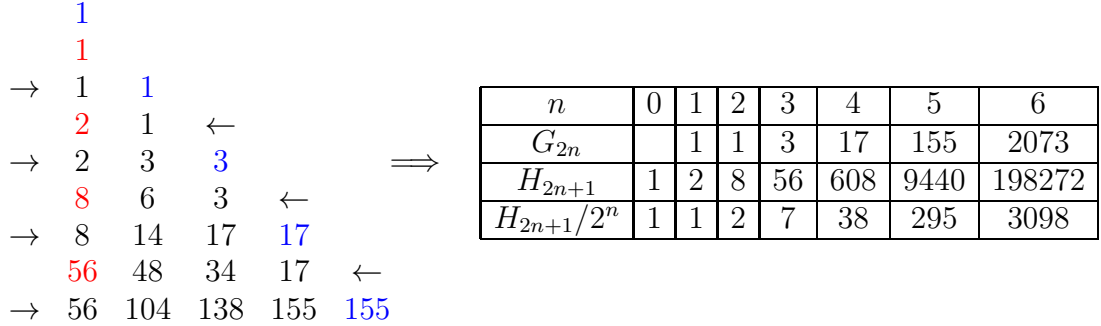

\begin{align*}
\begin{array}{cccccccc}
&\blue{1}&&&&&&\\
&\red{1}&&&&&&\\
\rightarrow&1&\blue{1}&&&&&\\
&\red{2}&1&\leftarrow&&&&\\
\rightarrow&2&3&\blue{3}&&&&\\
&\red{8}&6&3&\leftarrow&&&\\
\rightarrow&8&14&17&\blue{17}&&&\\
&\red{56}&48&34&17&\leftarrow&&\\
\rightarrow&56&104&138&155&\blue{155}&&\\
\end{array}\hspace{-1cm}
\Longrightarrow \quad
\begin{array}{|c|c|c|c|c|c|c|c|}
\hline
n&0&1&2&3&4&5&6\\
\hline
G_{2n}&&1&1&3&17&155&2073\\
\hline
H_{2n+1}&1&2&8&56&608&9440&198272\\
\hline
H_{2n+1}/2^{n}&{1}&{1}&{2}&{7}&{38}&{295}&{3098}\\
\hline
\end{array}
\end{align*}
\caption{The first values of the Genocchi numbers $G_{2n}$, median  Genocchi numbers $H_{2n+1}$ and normalized median Genocchi numbers $H_{2n+1}/2^n$ are 
tabulated in the right table.}
\end{figure}


In 2019 Hetyei~\cite{He19} introduced a hyperplane arrangement (called the homogenized Linial arrangement) and showed that its number of regions is a median Genocchi number. 
Lazar and Wachs \cite{LW19}  refined Hetyei's result by obtaining a combinatorial interpretation of the M\"obius function of this lattice in terms of varaiants of the Dumont permutations. 

A permutation $\sigma\in \S_{2n}$ is a \emph{D-permutation}  (resp. \emph{E-permutation}) 
if $i\leq \sigma(i)$ whenever $i$ is odd and
$i\geq \sigma(i)$ whenever $i$ is even (resp. 
if $i>\sigma(i)$ implies  $i$ is even and
$\sigma(i)$  $i$ is odd). An E-permutation is also called \emph{even-odd drop} permutation.
Introduce the set notations 
\begin{align*}
\D_{2n}&=\{\textrm{D-permutations on}\; [2n]\};\qquad 
\D\mathcal{C}_{2n}=\{\textrm{D-cycles on}\; [2n]\};\\
\E_{2n}&=\{\textrm{E-permutations on}\; [2n]\};\qquad 
\E\mathcal{C}_{2n}=\{\textrm{E-cycles on}\; [2n]\}.
\end{align*}
For example, 
\begin{multline*}
\D_{4}=\{(1)(2)(3)(4),\; (1,2)(3)(4), \; (1,4)(2)(3),\;
(3,4)(1)(2),\\
  (1,2)(3,4),\; (1,3,4)(2), \;(1,4,2)(3), (1,3,4,2)\};
  \end{multline*}
  and 
  \begin{multline*}
\E_{4}=\{(1)(2)(3)(4),\; (1,2)(3)(4), \; (1,4)(2)(3),\;
(3,4)(1)(2),\\
  (1,2)(3,4),\; (1,3,4)(2), \;(1,2,4)(3), (1,2,3,4)\}.
\end{multline*}
Here the permutations are written in cycle notations.

\begin{thm}[Lazar and Wachs \cite{LW19}] \label{thm:LW}
 For integer $n\geq 1$ we have 
\begin{align*}
H_{2n+1}&=|\D_{2n}|=|\E_{2n}|,\\
G_{2n}&=|\D\mathcal{C}_{2n}|.
\end{align*}
\end{thm}

Lazar and Wachs~\cite[Conjecture 6.5]{LW19} also make the following conjecture.

 \begin{conj}[Lazar and Wachs]\label{conj:LW}
 The number of $D$-permutations on $[2n]$  with $k$ 
 cycles equals the number of $E$-permutations on
 $[2n]$ with $k$ cycles for all $k$. Consequently
 $$
 G_{2n}=|\E\mathcal{C}_{2n}|.
 $$
 \end{conj}

In this paper we will take a different approach to 
enumerate the cycles and drops of the E-permutations and prove 
Lazar-Wachs' conjecture. A \emph{descent} of a permutation $\sigma$ on a finite set of positive integers is a pair $(\sigma_i, \sigma_{i+1})$ for which
$\sigma_i> \sigma_{i+1}$. 
We call $\sigma_i$ (resp. $\sigma_{i+1}$) the  descent top (resp. descent bottom). A descent pair  $(\sigma_i, \sigma_{i+1})$ is called \emph{even-odd} if $\sigma_i$ is even and $\sigma_{i+1}$ is odd. 
The number of descents of $\sigma$ is denoted by 
$\des\sigma$. Moreover,  the value $\sigma_i$ is said to be 
\begin{itemize}
\item a \emph{left-to-right maxima}, if $\sigma_j<\sigma_i$ for all $j<i$; 
\item a \emph{right-to-left minimum}, if $\sigma_i<\sigma_j$ for all $j>i$.
\end{itemize}
Let $\lma\,\sigma$ and 
 $\rmi\,\sigma$ be the numbers of left-to-right maxima and  right-to-left minima. 
 We further denote by $\lema\,\sigma$ and $\loma\,\sigma$  the numbers of \emph{even} and \emph{odd}
 left-to-right maxima. Similarly we define 
 $\remi\,\sigma$ and  $\romi\,\sigma$  to be 
 the numbers of \emph{even} and \emph{odd}
 right-to-left mnima.

Let $\X_{2n}$ be the set of  permutations in $\S_{2n}$ that contain only even-odd descents. A permutation in $\mathcal{X}_{2n}$ is called an \emph{even-odd descent permutation}.
For $n=1, 2$, we have 
 $\X_2=\{12, \red{2}1\}$ and 
$$
\X_4=\{1234,\; \red{2}134,\; 12\red{4}3,\; 3\red{4}12,\; \red{4}123,\; 23\red{4}1,\; 2\red{4}13,\; \red{2}1\red{4}3\}.
$$

Note that one version of Foata's first fundamental transformation $\varphi$
 takes drops to descents and maxima  of cycles to left-to-right maxima, i.e.,
 \begin{align}\label{cyc-lma}
(\drop, \cyc)\sigma=(\des, \lma)\varphi(\sigma)\quad \textrm{for}\; \sigma \in \S_n.
\end{align}
Here $\drop\,\sigma$ and $\cyc\,\sigma$ denote the numbers of drops and cycles of $\sigma\in \S_n$, respectively.
Hence, by restriction  $\varphi$ sets up  a bijection 
 from $\E_{2n}$ to $\X_{2n}$, and  the cardinality of $\X_{2n}$ equals the median Genocchi number $H_{2n+1}$.

An even-odd descent permutation  $\sigma\in \X_{2n}$ is called  
$\mathcal{E}$-permutation if for any integer $j\in \{1, \ldots, n\}$ the two entries $2j-1$ and $2j$  are not  simultanenously 
a descent bottom  and a descent top 
of $\sigma$.\footnote{ It is easy  to see that 
our definition of $\mathcal{E}$-permutations is equivalent to the \emph{primary even-odd descent permutations} in Eu et al. \cite{EFLL21}.}
Let $\overbar{\X}_{2n}$  be the set of $\mathcal{E}$-permutations 
in ${\X}_{2n}$ and $\overbar{\X}_{2n, k}$ be the subset consisting of permutations with $k$ descents.
The first few  $\mathcal{E}$-permutations  are
\begin{align*}
\overbar{\X}_{2,0}&=\{12\},\quad 
\overbar{\X}_{4,0}=\{1234\}, \quad
\overbar{\X}_{4,1}=\{3412,4123,2341,2413\}\\
\overbar{\X}_{6,0}&=\{123456\},\quad
\overbar{\X}_{6,1} =\{124563,124635,125634,126345,234156,241356,412356,\\
&\qquad 341256, 261345,361245, 461235,561234,236145,246135,256134,346125,356124,\\
&\qquad 456123,234615,235614,245613,345612,234561,612345\}.
\end{align*}
Note that  the even-odd descent permutation $\sigma=24163785$ is not an 
$\mathcal{E}$-permutation.
 
Eu-Fu-Lai-Lo~\cite[Theorem 1.2]{EFLL21}  recently studied the descent polynomials of
even-odd descent permutations, and proved, among other things,  the following result.
\begin{thm}[Eu et al. \cite{EFLL21}]\label{eu-thm} We have 
\begin{align}\label{eu-gamma}
X_n(t):=\sum_{\sigma\in \mathcal{X}_{2n}}t^{\des \sigma}=\sum_{k=0}^{\lfloor n/2\rfloor}|\overbar{\X}_{2n,k}|
 t^k(1+t)^{n-2k}.
\end{align}
\end{thm}

An expansion  of type \eqref{eu-gamma} is called the  $\gamma$-expansion of the polynomial $X_n(t)$ in the literature~\cite{At16,Br15}. Indeed,  
a polynomial with real coefficients $h(t)=\sum_{i=0}^nh_it^{i}$ is said to be \emph{palindromic} if 
 $h_i=h_{n-i}$ for $0\leq i\leq n/2$. It is 
 known that any palindromic polynomial can be written  uniquely in the form $\sum_{i=0}^{\lfloor(n-1)/2\rfloor}\gamma_it^{i}(1+t)^{n-2i}$. The coefficients $\gamma_i$ are called the $\gamma$-coefficients of $h(t)$. If the $\gamma$-coefficients $\gamma_i$ are all nonnegative then we say that $h(t)$ is  $\gamma$-positive. Note that the $\gamma$-positivity of 
 $h(t)$ implies the symmetry and unimodal property of the coefficients $(h_0, \ldots, h_n)$.
 A prototype of $\gamma$-positive polynomials is 
  the Eulerian polynomials  $A_n(t)$, which  are the descent polynomials of permutations in $\S_n$.
 The following expansion of Eulerian polynomials is  known~\cite{FS70, PZ21}:
 \begin{align}\label{fs-gamma}
 A_n(t):=\sum_{\sigma\in \S_n} t^{\des \sigma}=
 \sum_{k=0}^{\lfloor (n-1)/2\rfloor}2^k d_{n,k}\,t^k(1+t)^{n-1-2k},
 \end{align}
 where $d_{n,k}$ is the number of Andr\'e permutations in $\S_n$ with $k$ descents.
 
 This paper stems from the  observation that
 the generating function of the 
 descent polynomials $X_n(t)$ has a similar  continued fraction expansion as the Eulerian polynomials and  the cardinality $|\overbar{\X}_{2n,k}|$ in \eqref{eu-gamma} is divisible by  $4^k$, consequently \eqref{eu-gamma} would  imply a refinement of the normalized median Genocchi numbers
\begin{align}
 H_{2n+1}/2^n=\sum_{k=0}^{\lfloor n/2\rfloor}|\overbar{\X}_{2n,k}|/4^k.
\end{align}

In 2008 Br\"and\'en~\cite{Br08} studied a $(p,q)$-analogue of Eulerian polynomials and proved that 
 his generalized Eulerian polynomials have a
 $\gamma$-positive expansion by modifying Foata and Strehl's \emph{valley-hopping action}. 
  At the end of his paper Br\"and\'en~\cite{Br08}  speculated that 
  the corresponding $\gamma$-coefficients have a factor $(p+q)^k$, which is a $(p,q)$-analogue of $2^k$. Shin and Zeng~\cite{SZ12} confirmed   the conjectured  divisibility of $\gamma$-coefficients 
by using Flajolet-Viennot's combinatorial theory of orthogonal polynomials \cite{Fl80, viennot1984} and proved that the corresponding quotient $d_{n,k}(p,q)$ is a polynomial in $p$ and $q$ with nonegative integer coefficients. Finally, motivated by an open problem of Han~\cite{Han20} about   $q$-analogue of Euler numbers, the two authors \cite{PZ21} came up with a
 combinatorial interpretation for  $d_{n,k}(p,q)$ 
 by refining the Andr\'e permutation interpretation for $d_{n,k}$.

In this paper we shall  consider a natural $(p,q)$-analogue of the descent  polynomials $X_n(p,q,t)$ and  prove that the $\gamma$-coefficient $\gamma_{n,k}(p,q)$ of $X_n(p,q,t)$ has a factor  $(p+q)^{2k}$, and provide a combinatorial interpretation for $\gamma_{n,k}(p,q)/(p+q)^{2k}$, see Theorem~\ref{thm:main-pq}. It turns out that our approach can catch more permutation 
statistics over even-odd descent permutations. 
 For $j\in[n]$, the doubleton $\{2j-1,2j\}$ is called a \emph{domino} of  
 $\sigma\in\X_{2n}$ if  $2j$ and $2j-1$ are descent top and descent bottom of $\sigma$, respectively.
Denote by $\TB(\sigma)$  the number of dominos of $\sigma$.
For $\sigma=\sigma_1\sigma_2\ldots\sigma_n\in\S_n$, the statistic $\312\sigma$ is the number of pairs $(i,j)$ such that $2\leq i<j\leq n$ and $\sigma_{i-1}>\sigma_j>\sigma_i$.
The pair $(\sigma_{i-1}, \sigma_i)$ is called a \emph{left-embracing} of $\sigma_j$. Similarly, the statistic $\231\sigma$ is the number of pairs $(i,j)$ such that $1\leq i<j\leq n-1$ and $\sigma_j>\sigma_i>\sigma_{j+1}$. The pair $(\sigma_j, \sigma_{j+1})$ is called a 
 \emph{right-embracing} of $\sigma_i$.

We refine  the  median Genocchi numbers 
by  the octuple-variable  polynomials
\begin{align}\label{x-even-odd}
X_n:&=X_n(a,\bar a, b, \bar b, p,q,y, t)\nonumber\\
&=\sum_{\sigma\in \mathcal{X}_{2n}}
a^{\lema\,\sigma}\bar a^{\loma\,\sigma}b^{\romi\,\sigma}\bar b^{\remi\,\sigma}
p^{\231\sigma }q^{\312\sigma }y^{\TB\, \sigma}t^{\des \sigma},
\end{align}
which reduces to  $X_n(t)$ when $a=\bar a=b=\bar b=y=1$.
For example,  
\begin{align*}
X_1&=ab(\bar a\bar b+ty),\\
X_2&=a^2\bar a^2 b^2\bar b^2+2a^2\bar ab^2\bar b ty
+a\bar a b\bar bpqt+ab^2\bar bq^2t+a^2\bar ab p^2t+a^2b^2pqt+a^2b^2t^2y^2.
\end{align*}
We will prove  an explicit J-fraction expantion (Theorem~\ref{master-cf}) for the ordinary  generating function of $X_n$, and prove $(p,q)$-analogue of Theorem~\ref{eu-thm} and 
 Lazar-Wachs' conjecture as applications.

The Genocchi numbers $G_{2n}$ can also be defined by  the exponential generating function
$
\sum_{n=1}^{\infty}G_{2n}\frac{x^{2n}}{(2n)!}=x\tan \frac{x}{2},
$
while
the median Genocchi numbers do not seem to have a raisonable  exponential generating function. Our method 
 relies on the   S-fraction expansions of 
the ordinary generating functions  of Genocchi and median Genocchi numbers (see \cite{Vi82, DZ94}):
 \begin{subequations}
\begin{align}
\sum_{n=0}^{\infty}G_{2n+2}x^{n}=\cfrac{1}{
1  - \cfrac{1^2\cdot  x}{
1 - \cfrac{1\cdot 2\cdot  x}{
1-\cfrac{2^2\cdot x}{
1-\cfrac{2\cdot 3\cdot x}{
\cdots}}}}}=1+x+3x^2+17 x^3+\cdots,\label{v-genocchi}\end{align}
\begin{align}
1+\sum_{n=0}^{\infty}H_{2n+1}x^{n+1}&=\cfrac{1}{
1  - \cfrac{1^2\cdot  x}{
1 - \cfrac{1^2\cdot  x}{
1-\cfrac{2^2\cdot x}{
1-\cfrac{2^2\cdot x}{
\cdots}}}}}=1+x+2x^2+8x^3\cdots.\label{v-mgenocchi}
\end{align}
\end{subequations}
By the general theory of orthogonal polynomials the above 
continued fraction expansions mean that the Genocchi and median Genocchi numbers are moments of orthogonal polynomials. More precisely, as shown in \cite{GZ21},
they are indeed moments of special or shifted continuous dual 
Hahn polynomials, see  also 
\cite{CKS16, CJW11, BS21,GZ21} for recent papers on combinatorial aspects 
 related to the  moments 
of Askey-Wilson  polynomials.

For reader's convenience, we recall two  standard \emph{contraction formulae} transforming an S-fraction to 
J-fractions, see \cite{DZ94}. 
\begin{lem}[Contraction formula]\label{contra-formula}
For any sequence $\{\alpha_n\}$ of elements in an arbitrary ring containing $\Z$, the following holds
\begin{align}\label{contraction1}
\cfrac{1}{
1-\cfrac{\alpha_{1}x}{1-\cfrac{\alpha_{2}x}{
1-\cfrac{\alpha_{3}x}{
{\ddots}}}}}
=\cfrac{1}{1-b_0x-\cfrac{\la_1x^{2}}{1-b_1x-\cfrac{\la_2x^{2}}{1-b_2 x-\cfrac{\la_3z^2}{\ddots}}}}
\end{align}
if 
$
b_0=\alpha_1, \quad 
b_n=\alpha_{2n}+\alpha_{2n+1},\quad
\lambda_n=\alpha_{2n-1}\alpha_{2n}\quad \textrm{for}\; n\geq 1;
$
and 
\begin{align}\label{contraction2}
\cfrac{1}{
1-\cfrac{\alpha_{1}x}{1-\cfrac{\alpha_{2}x}{
1-\cfrac{\alpha_{3}x}{
{\ddots}}}}}
&=1+\cfrac{\a_{1}x}{1-b_0x-\cfrac{\la_1 x^{2}}{1-b_1x-\cfrac{\la_2x^{2}}{1-b_3 x-\cfrac{\la_3x^2}{\ddots}}}}.
\end{align}
if 
$
b_{n-1}=\alpha_{2n-1}+\a_{2n},\quad
\lambda_n=\alpha_{2n}\alpha_{2n+1}\quad \textrm{for}\; n\geq 1.$
\end{lem}

 From \eqref{v-mgenocchi} we  derive the  J-fraction
\begin{align}\label{J-frac-MG}
\sum_{n=0}^{\infty}H_{2n+1}x^{n}=\cfrac{1}{
1  -2\cdot 1^2 \;x-\cfrac{1^2\cdot 2^2\; x^2}{
1 - 2\cdot 2^2 \;x-\cfrac{2^2\cdot 3^2\;  x^2}{
1-2\cdot 3^2 \;x-\cfrac{3^2 \cdot 4^2\;  x^2}{
\cdots}}}}.
\end{align}
Let $h_{n}=H_{2n+1}/2^n$ ($n\in \N$) be the normalized median Genocchi numbers.
By  replacing  $x$ by $x/2$, 
 we obtain
\begin{align}\label{J-frac-NMG}
\sum_{n=0}^{\infty}h_{n}x^{n}=\cfrac{1}{
1  -1^2\; x-\cfrac{1^2\cdot  x^2}{
1 - 2^2 \;x-\cfrac{{3\choose 2}^2 \cdot x^2}{
1-3^2 \;x-\cfrac{{4\choose 2}^2\cdot x^2}{
\cdots}}}}.
\end{align}

Dumont \cite{Du74} initiated the combinatorial characterization of Genocchi numbers.
A  permutation $\sigma\in \S_{2n}$ is a \emph{Dumont permutation}  if 
$\sigma(2i-1)\geq 2i-1$ and $\sigma(2i)<2i$ for $i\in [n]$.
A  permutation $\sigma\in \S_{2n}$ is a \emph{Dumont derangement} 
if $\sigma(2i-1)>2i-1$ and $\sigma(2i)<2i$ for $i\in [n]$.
Dumont~\cite{Du74} and Dumont-Randrianarivony~\cite{DR94}
proved respectively that the Dumont permutations in $\S_{2n}$ is the Genocchi number $G_{2n}$ and the Dumont derangements in $\S_{2n}$ is the median Genocchi number $H_{2n+1}$.
Kitaev and Remmel \cite{KR07} conjectured  and Burstein, Josuat-Verg\`es, and Stromquist \cite{BJS10} proved
that the set of   permutations in $\S_{2n}$  with only even-even descents has cardinality equal to $G_{2n}$. Clearly the set of even-even descent permutations in $\S_{2n}$ is in bijection with the set of odd-odd descent permutations in $\S_{2n-1}$. 

Let $\Y_{2n+1}$ be the set of permutations in $\S_{2n+1}$ that contain only odd-odd descents, notice that $|\Y_{2n+1}|=G_{2n+2}$. 
Let $\Y_{2n+1}^*$ be the subset of $\Y_{2n+1}$ consisting of the permutations ending with an odd element. The first three sets are $\Y_{1}^*=\{1\}$,
$\Y_{3}^*=\{123, 231\}$ and
$$
\Y_{5}^*=\{12345, 12453, 23451, 24531, 24513, 23145, 31245, 45123\}.
$$
Eu et al. \cite{EFLL21} proved that the cardinality of $\Y_{2n+1}^*$ is equal to the median Genocchi 
number $H_{2n+1}$. They also considered  a $q$-analogue of the descent polynomials of $\Y_{2n+1}^*$ and proved a similar $\gamma$-expansion for the $q$-descent polynomials. 
We  show Section~7 that it is possible to  refine their results and prove similar results as for the descent polynomials of even-odd descent permutations. 

\section{Main results}

For any integer $n\geq 1$ we define the 
generalized $(p,q)$-analogue of $n$  by
\begin{subequations}
\begin{align}
[x,n,y]_{p,q}&=xp^{n-1}+\sum_{i=1}^{n-2}p^{n-1-i}q^i+yq^{n-1},\\
[x,n]_{p,q}&=xp^{n-1}+\sum_{i=1}^{n-1}p^{n-1-i}q^i,
\end{align}
\end{subequations}
with  $[x,1,y]_{p,q}=xy$ and   $[x,1]_{p,q}=x$. In particular,
$$
[n]_{p,q}=[1,n]_{p,q}=\frac{p^n-q^n}{p-q},
$$
and the $(p,q)$-analogue of binomial coefficient $\binom{n}{k}$ is defined by
$$
\binom{n}{k}_{p,q}=\frac{[n]_{p,q}\ldots [n-k+1]_{p,q}}{[1]_{p,q}\ldots [k]_{p,q}}\qquad (1\leq k\leq n).
$$

\begin{thm}\label{master-cf} We have 
\begin{multline}\label{cf:X}
1+\sum_{n=1}^{\infty}X_n(a,\bar a, b, \bar b, p,q,y, t)x^{n}=
\frac{1}{1-ab(\bar a  \bar b+ty)\,x-\cfrac{abt( a p+\bar b q)(\bar a p+bq)x^2}{\ddots}}
\end{multline}
with coefficients for $n\geq 1$
\begin{align*}\begin{cases}
b_{n-1}&=[\bar a, n, b]_{p,q}[a,n, \bar b]_{p,q}+
ty \,[a,n]_{p,q}[b,n]_{q,p},\\
\lambda_{n}&=t[a,n]_{p,q}\,[a,n+1, \bar b]_{p,q}\,[b, n]_{q,p}\,[\bar a,n+1,b]_{p,q}.
\end{cases}
\end{align*}
\end{thm}
We first record some special cases of Theorem~\ref{master-cf}.
When $b=\bar b=1$ we have the more compact formula.
\begin{cor}\label{cor-master-cf}
\begin{multline}\label{cf:X}
1+\sum_{n=1}^{\infty}X_{n}(a,\bar a,1,1, p,q, y, t)x^{n}=\\
\frac{1}{1-(\bar a+ty)ax-\cfrac{at(\bar a p+q)(ap+q)x^2}
{1-([\bar a, 2]_{p,q}+ty[2]_{p,q})[a,2]_{p,q}x-\cfrac{t[a,2]_{p,q}[a,3]_{p,q}[2]_{p,q}[\bar a,3]_{p,q}x^2}{\cdots}}}
\end{multline}
with coefficients for $n\geq 1$
\begin{align*}\begin{cases}
b_{n-1}&=([\bar a, n]_{p,q}+
ty \,[n]_{p,q})[a,n]_{p,q},\\
\lambda_{n}&=t[a,n]_{p,q}\,[a,n+1]_{p,q}\,[n]_{p,q}\,[\bar a,n+1]_{p,q}.
\end{cases}
\end{align*}
\end{cor}

Let  $y=1$ in Theorem~\ref{master-cf}, by the contraction formula we can rewrite the above J-fraction as an S-fraction.
\begin{cor} We have the S-fraction expansion 
\begin{equation}\label{Scf:z}
1+a\bar a\sum_{n=0}^{\infty}X_{n}(a,\bar a,1, 1,  p,q, 1, t)x^{n+1}=
\cfrac{1}{1-\cfrac{a\bar a\cdot x}{1-\cfrac{at\cdot x}{1-\cfrac{[\bar a,2]_{p,q}[a,2]_{p,q}\cdot x}
{1-\cfrac{[a,2]_{p,q}[2]_{p,q}\,t\cdot x}{\cdots}}}}}
\end{equation}
with $X_0=1$ and coefficients 
\begin{align*}
\begin{cases}
\alpha_{2n-1}&=[\bar a,n]_{p,q}[a,n]_{p,q},\qquad n\geq 1,\\
\alpha_{2n}&=[a,n]_{p,q}[n]_{p,q} \,t.
\end{cases}
\end{align*}
\end{cor}

By Foata's fundamental transformation (see \eqref{cyc-lma}) we have 
$$
P_n(t, z):=\sum_{\sigma\in \X_{2n}}z^{\des \sigma}z^{\lma \sigma}=\sum_{\sigma\in \E_{2n}}z^{\drop \sigma}z^{\cyc \sigma}. 
$$
Thus, 
letting $a=\bar a=z$, $p=q=y=1$ in Corollary~\ref{cor-master-cf} we derive the continued fraction for the ordinary  generating function of 
$P_n(t,z)$.
\begin{cor}\label{cf-e-cyc}
\begin{align}
&1+\sum_{n=1}^\infty P_n(t, z)x^n\nonumber\\
&=\cfrac{1}{1-z(z+t)\,x-
\cfrac{z(z+1)^2t\,x^2}{\cfrac{\ddots}{1-(z+n)(z+n+(n+1)t)\,x-
\cfrac{(n+1)(z+n)(z+n+1)^2t\,x^2}{\cdots}}}}.
\end{align}
\end{cor}

Taking  $y=0$ in  Theorem~\ref{master-cf}  we obtain the corresponding formula for $\mathcal{E}$-permutations.
\begin{cor}\label{cor:gammacf} We have
\begin{multline}
1+\sum_{n=1}^\infty\sum_{\sigma\in \overbar{\X}_{2n}}a^{\lema\,\sigma}
{\bar a}^{\loma\,\sigma}
p^{\231\sigma }q^{\312\sigma }t^{\des \sigma}x^{n}
=\\
\cfrac{1}{
1  -a{\bar a}\, x-\cfrac{at[2]_{p,q}[a,2]_{p,q}\, x^2}{
1 - (ap+q)({\bar a}p+q) x-\cfrac{t[2]_{p,q}[a,2]_{p,q}[3]_{p,q}[\bar a,3]_{p,q}\, x^2}{
\cdots}}}.
\end{multline}
with coefficients for $n\geq 1$
\begin{align*}\begin{cases}
b_{n-1}&=[a, n]_{p,q}[\bar a, n]_{p,q},\\
\lambda_{n}&=t[a,n]_{p,q}\,[a,n+1]_{p,q}\,[n]_{p,q}\,[\bar a,n+1]_{p,q}.
\end{cases}
\end{align*}
\end{cor}

We shall prove Theorem~\ref{master-cf} in Section~3 by constructin 
a bijection from the even-odd descent permutations to  
lattice path diagrams.  Combining Theorem~\ref{master-cf}, the surjective pistol theory developped by Dumont-Randrianarivony and Lazar-Wachs and 
 Randrianarovony-Zeng's  continued fraction expansion for the ordinary generating function of generalized Dumont-Foata polynomials
 we prove  the  conjecture of 
Lazar-Wachs~\cite[Conjecture 6.4]{LW19} in section 4.
 \begin{thm}\label{LW-conj} Conjecture \ref{conj:LW} is true, i.e., for 
 all $n\geq 1$, 
\begin{align}\label{D-E-cycle}
\sum_{\sigma\in \E_{2n}} z^{\cyc \sigma}
=\sum_{\sigma \in \D_{2n}}z^{\cyc(\sigma)}.
\end{align}
Consequently 
$|\E\mathcal{C}_{2n}|=|\D\mathcal{C}_{2n}|=G_{2n+2}$.
 \end{thm}

Our third  result  is a quasi-gamma decomposition of $X_n(a, 1,b, 1,p,q,y, t)$, which provides 
a  neat $(p,q)$-analogue of  
the $\gamma$-formula \eqref{eu-gamma}  by setting  $a=b=y=1$.
 
\begin{thm}\label{thm:gamma}
We have
\begin{multline}\label{cf:ab}
1+\sum_{n=1}^{\infty}X_n(a,1, b, 1, p,q,y, t)x^{n}=
\frac{1}{1-ab(\bar a  \bar b+ty)\,x-\cfrac{abt(\bar a p+\bar q)(\bar a p+bq)x^2}{\ddots}}
\end{multline}
with coefficients for $n\geq 1$
\begin{align*}\begin{cases}
b_{n-1}&=(1+
ty )[a,n]_{p,q}[b, n]_{q,p},\\
\lambda_{n}&=t[a,n]_{p,q}\,[a,n+1]_{p,q}\,[b, n]_{q,p}\,[b,n+1]_{q,p}.
\end{cases}
\end{align*}
Moreover the following $\gamma$-formula holds
\begin{subequations}
\begin{align}\label{formula:gamma}
X_n(a,1,b, 1, p,q,y, t)=\sum_{k=0}^{\lfloor n/2\rfloor}\gamma_{n,k}(a, b, p,q)\, t^k(1+yt)^{n-2k},
\end{align}
{\rm with  coefficients}
\begin{align}\label{coeff:gamma}
\gamma_{n,k}(a,b, p,q)=\sum_{\sigma\in\overline{\X}_{2n,k}}a^{\lema\,\sigma}b^{\romi\,\sigma}p^{\231\sigma}q^{\312\sigma}.
\end{align}
\end{subequations}
\end{thm}
Clearly Theorem~\ref{master-cf} reduces to \eqref{cf:ab} by setting $\bar a=\bar b=1$. By comparing  Corollary \ref{cor:gammacf}
with formula \eqref{coeff:gamma} we derive \eqref{formula:gamma}.
We shall give another proof of \eqref{formula:gamma} using \emph{Inter-hopping} action on $\X_{2n}$ in Section~5.

A \emph{weak signature} of order $k$ in $[2n]$ is a subset $S\subseteq [2n]$ consisting of 
$k$ odd and $k$ even  numbers   such that 
for each $i\in[2n]$, 
the number of odd  numbers is greater than or equal to the number of even numbers in 
$S\cap \{1, \ldots, i\}$. Note that our weak signature
is different with  the \emph{signature} in \cite{BJS10, EFLL21}. Denote by $S_o$ (resp. $S_e$) the set of odd (resp. even ) elements of $S$.
Let $S_o(i)$ and $S_e(i)$ be the numbers of odd and even elements in $S$ less than $i$.
The associated signature function $s$ has domain $[2n]$ and is defined by
\begin{align}\label{def-s}
s(x)=S_o(x)-S_e(x)+(1 \; \textrm{if $x\in S_o$ or $x\notin S$}).
\end{align}
We  use  $l_{\sigma}(i)$ (resp. $r_{\sigma}(\sigma_i)$) to denote the number of left-embracings (resp. right-embracings) of $\sigma_i$ in $\sigma$, thus
 \begin{subequations}
 \begin{align}
 \312\sigma&=\sum_{i=1}^n l_{\sigma}(\sigma_i);\label{embracea}\\
 \231\sigma&=\sum_{i=1}^n r_{\sigma}(\sigma_i).\label{embraceb}
 \end{align}
 \end{subequations}

\begin{lem}\label{lem:signature}
For $\sigma\in \X_{2n}$, if $S$ be the set of descent tops and bottoms of $\sigma$, then $S$ is a week signature of $[2n]$ and the associated signature function $s$ is given by
\begin{align}\label{lem-s}
s(i)=l_{\sigma}(i)+r_{\sigma}(i)+1\qquad \forall i\in [2n].
\end{align}
\end{lem}
\begin{proof}
Let  $\sigma\in \X_{2n}$.  
For any $i\in [2n]$, there is an injection $\phi: S_e\cap [i]\to S_o\cap [i]$ which maps each decent top $t\in S_e\cap [i]$ to a descent bottom $b\in S_o\cap [i]$ such that $(t, b)$ forms a descent pair of $\sigma$.  Hence $S_o(i)\geq S_e(i)$ 
and  $S$ is a week signature of $[2n]$.
 
Next we compute the embracing numbers  
$l_{\sigma}(i)+r_{\sigma}(i)$ of $i\in [2n]$.
 \begin{enumerate}
\item If  $i\in S$ is  odd or $i\notin S$, 
the number of even-odd descents $(t,b)$ with $t<i$ is equal to $S_e(i)$,  so the embracing number of $i$ is $S_o(i)-S_e(i)$.
\item
If  $i\in S$ is  even, the number of even-odd descents $(t,b)$ 
with $t\leq i$ is equal to $S_e(i)+1$, 
 so $i$ has $S_o(i)-S_e(i)-1$ embraces.
\end{enumerate}
Combining the two cases and comparing with \eqref{def-s} we obtain \eqref{lem-s}.
\end{proof}

In what follows, for any $\sigma\in\X_{2n}$ we use $S_{\sigma}$ and 
 $s_{\sigma}$  to denote 
the weak signature of $\sigma$  and the associated signature function, respectively.

\begin{defi}\label{np}
An $\E$-permutation 
$\sigma\in\overbar{\X}_{2n}$ is  normalized 
if it satisfies   the  following conditions: 
 for $j\in [n]$,  
\begin{itemize}
\item[(a)] if $(2j-1, s(2j-1))\in S_{\sigma}\times (2\N+1)$, 
then $l_\sigma(2j)$ is even;
\item[(b)] if $(2j-1, s(2j-1))\in S_{\sigma}\times 2\N$,
then  $l_\sigma(2j-1)$ is even;
\item[(c)] if $(2j, s(2j))\in S_{\sigma}\times (2\N+1)$, then 
$r_\sigma(2j-1)$ is even;
\item[(d)]  if $(2j, s(2j))\in S_{\sigma}\times 2\N$, then 
$r_\sigma(2j)$ is even.
\end{itemize}
\end{defi}
Let $\widehat{\X}_{2n}$ be the set of  normalized  $\mathcal{E}$-permutations 
in $\overbar{\X}_{2n}$ and let 
$$
\widehat{\X}_{2n,k}=\{\sigma\in \widehat{\X}_{2n}: \des \sigma=k\}.
$$
\begin{exam}
The following are the first few normalized $\mathcal{E}$-permutations:
\begin{align*}
\widehat{\X}_2&=\{1\,2\}\\
\widehat{\X}_4&=\{1\,2\,3\,4, \; 2\,4\,1\,3\}\\
\widehat{\X}_6&=\{1\,2\,3\,4\,5\,6,\;1\,2\,4\,6\,3\,5,\; 2\,4\,1\,3\,5\,6,\; 2\,3\,4\,6\,1\,5,\; 2\,6\,1\,3\,4\,5,\; 2\,3\,6\,1\,4\,5,\; 2\,4\,6\,1\,3\,5\}.
\end{align*}
\end{exam}
\begin{prop}
For $\sigma\in\widehat{\X}_{2n,k}$, we have $\312\sigma\geq k$ and $\231\sigma\geq k$.
\end{prop}
\begin{proof}
As $\sigma\in\widehat{\X}_{2n,k}$, according to the definition of normalized $\mathcal{E}$-permutations, 
\begin{itemize}
\item[(a)] if $2i-1\in S_{\sigma}$ and $s_{\sigma}(2i-1)$ is odd, then $s_{\sigma}(2i)=s_{\sigma}(2i-1)+1$, so by Lemma~\ref{lem:signature}, we have $r_{\sigma}(2i)$ is odd.
\item[(b)] if $2i-1\in S_{\sigma}$ and $s_{\sigma}(2i-1)$ is even, then $l_{\sigma}(2i-1)$ is even, by Lemma~\ref{lem:signature}, $r_{\sigma}(2i-1)$ is odd.
\item[(c)] if $2i\in S_{\sigma}$ and $s_{\sigma}(2i)$ is odd then $r_{\sigma}(2i-1)$ is even, and $s_{\sigma}(2i-1)=s_{\sigma}(2i)+1$ is even, by Lemma~\ref{lem:signature}, $l_{\sigma}(2i-1)$ is odd.
\item[(d)] if $2i\in S_{\sigma}$ and $s_{\sigma}(2i)$ is even, then $\gamma_{\sigma}(2i)$ is even, also by Lemma~\ref{lem:signature}, we have $l_{\sigma}(2i)$ is odd.
\end{itemize}
So, for a descent bottom $2i-1$ of $\sigma$, either $2i-1$ or $2i$ contributes at least one $\231$ pattern. And for a descent top $2j$ of $\sigma$, either $2j$ or $2j-1$ contributes at least one $\312$ pattern. 
Therefore, $\312\sigma\geq k$ and $\231\sigma\geq k$.
\end{proof}

\begin{subequations}
\begin{thm}\label{thm:main-pq}
We have the $(p,q)$-analogue of \eqref{eu-gamma}
\begin{align}\label{pq-formula:gamma}
X_n(p,q,t):=\sum_{\sigma\in \X_{2n}} 
p^{\231\sigma}q^{\312\sigma}t^{\des \sigma}=\sum_{k=0}^{\lfloor n/2\rfloor}\gamma_{n,k}(p,q)\, t^k(1+t)^{n-2k},
\end{align}
{\rm with  coefficients}
\begin{align}\label{pq-coeff:gamma}
\gamma_{n,k}(p,q)=\sum_{\sigma\in\overline{\X}_{2n,k}}p^{\231\sigma}q^{\312\sigma}.
\end{align}
Moreover, the $\gamma$-coefficients $\gamma_{n,k}(p,q)$ have the factorization 
\begin{align}\label{x-gamma}
\gamma_{n,k}(p,q)=
(p+q)^{2k}\sum_{\sigma\in\widehat{\X}_{2n,k}}q^{\312\sigma-k}p^{\231\sigma-k}.
\end{align}
\end{thm}
\end{subequations}
Clearly Theorem~\ref{thm:gamma} implies \eqref{pq-formula:gamma}.
We shall prove Theorem~\ref{thm:main-pq} \eqref{x-gamma} in Section~6.
From Theorem~\ref{thm:main-pq} and  Corollary~\ref{cor:gammacf}  we derive the following J-fraction.
\begin{cor} We have 
\begin{multline}\label{p3}
\sum_{n\geq0}\sum_{\sigma\in \widehat{\X}_{2n}}p^{\231\sigma }q^{\312\sigma }t^{\des \sigma}x^{n}\\
=\cfrac{1}{
1  -[1]_{p,q}^2 x-\cfrac{{2\choose 2}_{p,q}^2\,t\cdot  x^2}{
1 - [2]_{p,q}^2 x-\cfrac{{3\choose 2}_{p,q}^2\,t\cdot  x^2}{
1-[3]_{p,q}^2 x-\cfrac{{4\choose 2}_{p,q}^2\,t\cdot x^2}{
\cdots}}}}.
\end{multline}
\end{cor}
Note that \eqref{p3} reduces to \eqref{J-frac-NMG} when $p=q=t=1$.
  Thus $|\widehat{\X}_{2n}|=h_n$ for $n\geq 1$.

\section{Even-odd permutations and 
labelled Motzkin paths 
}
A \emph{Motzkin path}  of length  $n$ is a sequence of points 
$\omega:=(\omega_0, \ldots, \omega_n)$  in the integer plan 
$\mathbb{Z}\times\mathbb{Z}$  such that 
\begin{itemize}
\item $\omega_0=(0,0)$ and $\omega_n=(n,0)$,
\item $\omega_i-\omega_{i-1}\in \{(1,0), (1, 1), (1,-1)\}$,
\item $\omega_i:=(x_i, y_i)\in \N\times \N$ for $i=0,\ldots, n$.
\end{itemize}
In other words, a Motzkin path of length $n\geq 0$ is a lattice path in the 
right quadrant $\N\times \N$ starting at $(0,0)$ and  ending at $(n,0)$, each step $s_i=\omega_i-\omega_{i-1}$ is a 
rise $\mathsf{U}=(1,1)$, fall
$\mathsf{D}=(1,-1)$ or level $\mathsf{L}=(1,0)$. A 2-Motzkin path is a Motzkin path with two types of level-steps $\mathsf{L}_1$ and $\mathsf{L}_2$.
Let $\mathcal{MP}_n$ be the set of 2-Motzkin paths of 
length $n$.
Clearly we can  identify 
 2-Motzkin paths of length $n$ with words $w$ on 
 $\{\mathsf{U, L_1, L_2, D}\}$ of length $n$ such that all prefixes of $w$ 
 contain no more $\mathsf{D}$'s than $\mathsf{U}$'s and the number of $\mathsf{D}$'s equals the number of $\mathsf{D}'s$.
 The \emph{height} of $\omega_i$ is 
 the ordinate of $\omega_i$ and denoted by $h(\omega_i)$, which is also called the height 
 of the step $s_i$.

Let  $\mathbf{a}=(a_i)_{i\geq 0}$,
$\mathbf{b}=(b_i)_{i\geq 0}$, $\mathbf{b'}=(b'_i)_{i\geq 0}$  and $\mathbf{c}=(c_i)_{i\geq 1}$  be  four  sequences of indeterminates; we will work in the ring $\Z[[\mathbf{a}, \mathbf{b}, \mathbf{b'}, \mathbf{c}]]$. 
To each Motzkin path $\omega$ we assign a weight $W(\omega)$ that is the product of the weights for the individual steps, where 
a up-step (resp. down-step) at height 
$i$ gets weight  $a_i$ (resp. $\lambda_i$), and 
a level-step of type 1 (resp. 2) at height $i$ gets weight 
$b_i$ (resp.  $b'_i$). 
The following  result of 
 Flajolet~\cite{Fl80} is folklore.
\begin{lem}[Flajolet]\label{flajolet} We have 
$$
\sum_{n=0}^{\infty}\left(\sum_{\omega\in \mathcal{MP}_n}W(\omega)\right)x^n
=\cfrac{1}{1-(b_0+b_0')x-\cfrac{a_0c_1x^2}{1-(b_1+b_1')x-\cfrac{a_1c_2x^2}{1-(b_2+b_2')x-\cdots}}}.
$$
\end{lem}

\begin{defi}
A path diagram of length $n$ is a triplet $(\omega,(\xi,\xi'))$, where $\omega$ is a Motzkin path of length $n$,  $\xi=(\xi_1,\ldots,\xi_n)$ and $\xi'=(\xi'_1,\ldots,\xi'_n)$ are two integer sequences satisfying  the following conditions:
\begin{subequations}
\begin{itemize}
\item 
If the $k$-th step of $\omega$ is a rise with  height $h\geq 0$, then
\begin{align}\label{case1}
(\xi_i,\xi_i')\in \{0, \ldots, h\}\times \{0, \ldots, h+1\}.
\end{align}
\item
If the $k$-th step of $\omega$ is a fall with  height $h>0$, then
\begin{align}\label{case2}
(\xi_i,\xi_i')\in \{0, \ldots, h\}\times \{0, \ldots, h-1\}.
\end{align}
\item
If the $k$-th step of $\omega$ is a level  of type 1 or 2 with 
 height $h\geq 0$, then
\begin{align}\label{case3}
(\xi_i,\xi_i')\in \{0, \ldots, h\}\times \{0, \ldots, h\}.
\end{align}
\end{itemize}
\end{subequations}
We denote by $\mathcal{PD}_n$ the set of path diagrams  of length $n$. 
\end{defi}

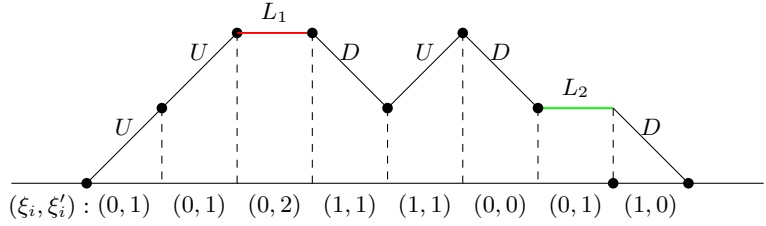
\begin{figure}[t]
\begin{tikzpicture}
\tikzstyle{every node}=[font=\scriptsize]
\draw (0,0)--(1,1);
\fill (0,0) circle (2pt);
\fill (1,1) circle (2pt);
\draw (1,1)--(2,2);
\fill (2,2) circle (2pt);
\draw[line width=0.25mm, red] (2,2)--(3,2);
\fill (3,2) circle (2pt);
\draw (3,2)--(4,1);
\draw (4,1)--(5,2);
\draw (5,2)--(6,1);
\draw (7,1)--(8,0);
\draw[line width=0.25mm, green] (6,1)--(7,1);
\fill (4,1) circle (2pt);
\fill (5,2) circle (2pt);
\fill (6,1) circle (2pt);
\fill (7,0) circle (2pt);
\fill (8,0) circle (2pt);
\draw (-1,0)--(9,0);
\draw[dashed] (1,1)--(1,0);
\draw[dashed] (2,2)--(2,0);
\draw[dashed] (3,2)--(3,0);
\draw[dashed] (4,1)--(4,0);
\draw[dashed] (5,2)--(5,0);
\draw[dashed] (6,1)--(6,0);
\draw[dashed] (7,1)--(7,0);
\node[below] at (-0.5,0) {$(\xi_i,\xi_i'):$};
\node[below] at (0.5,0) {$(0,1)$};
\node[below] at (1.5,0) {$(0,1)$};
\node[below] at (2.5,0) {$(0,2)$};
\node[below] at (3.5,0) {$(1,1)$};
\node[below] at (4.5,0) {$(1,1)$};
\node[below] at (5.5,0) {$(0,0)$};
\node[below] at (6.5,0) {$(0,1)$};
\node[below] at (7.5,0) {$(1,0)$};
\node[above] at (2.5,2) {$L_1$};
\node[above] at (6.5,1) {$L_2$};
\node[above] at (0.5,0.5) {$U$};
\node[above] at (1.5,1.5) {$U$};
\node[above] at (3.5,1.5) {$D$};
\node[above] at (4.5,1.5) {$U$};
\node[above] at (5.5,1.5) {$D$};
\node[above] at (7.5,0.5) {$D$};
\end{tikzpicture}\caption{Illustration of $\Phi: \sigma\mapsto (\omega,(\xi,\xi'))$.}\label{F1}
\end{figure}

For $\sigma\in\X_{2n}$, we construct the path diagram $\Phi(\sigma)=(\omega,(\xi,\xi'))$ in  $\mathcal{PD}_n$
as in the following. 
Let
 $\overline{S}_{\sigma}=[2n]\setminus S_{\sigma}$.
For $j\in \{1, \ldots, n\}$ we define  the $j$-th step $s_j=\omega_j-\omega_{j-1}$ of 
the path
$\omega=(\omega_0, \omega_1,\ldots, \omega_n)$  by 
\begin{subequations}
\begin{equation}\label{MP}
s_j=\begin{cases}
U& \textrm{if}\quad (2j-1, 2j)\in S_{\sigma}\times {\overline S}_{\sigma}\\
D&\textrm{if}\quad (2j-1, 2j)\in 
{\overline S}_{\sigma}\times S_{\sigma}\\
L_1&\textrm{if}\quad 
(2j-1, 2j)\in 
{\overline S}_{\sigma}\times {\overline S}_{\sigma}\\
L_2&\textrm{if}\quad (2j-1, 2j)\in 
 S_{\sigma}\times S_{\sigma}
\end{cases}
\end{equation}
and  the bi-sequence $(\xi, \xi')$ by
\begin{align}\label{xi-s}
(\xi_j, \xi'_j)=(r_\sigma(2j-1), r_\sigma(2j)).
\end{align}
\end{subequations}

The \emph{restriction} of 
$\sigma$ on $[2j]$ is the  word $\pi_j$ obtained from $\sigma$ by
replacing each subword of $\sigma$ consisting of consecutive letters 
 greater than $2j$ by a \emph{slot} $\x$. 
It is easier to  describe the above  construction 
 by looking at the successive restrictions 
of $\sigma$ on $\{1, \ldots, 2j\}$ for $j\in [n]$. 
 For example, if
 $\sigma=2\;6\;8\cdot 1\;4\;7\;14\cdot 9\;10\;12\cdot 3\;5\;11\;15\;16\cdot 13\in\X_{16}$, then $S=\{1,3,9,13\}\cup\{8,12,14,16\}$. Hence $\pi_0=\x$,  the successive restrictions of $\sigma$ read as follows:
 

\begin{figure}[h]
\begin{tabular}{c|l|c|c|c}
$i$&$\pi_j$&$(2j-1, 2j)$&$(\xi_j,\xi_j')$&$(s_{\sigma}(2j-1), s_{\sigma}(2j))$\\\hline
\\[-1em]
1&$2\;\x \;1\;\x$& $(1,2)\in S_{\sigma}\times {\overline S}_{\sigma}$& (0,1)&(1,2)\\\hline
\\[-1em]
2&$2\;\x\; 1\;4\;\x\;3\; \x$& $(3,4)\in S_{\sigma}\times {\overline S}_{\sigma}$&(0,1)&(2,3)\\\hline
\\[-1em]
3&$2\;6\;\x\; 1\;4\;\x\;3\; 5\;\x$& $(5,6)\in {\overline S}_{\sigma}\times {\overline  S}_{\sigma}$&(0,2)&(3,3)\\
\hline
\\[-1em]
4&$2\;6\;8\;1\;4\;7\;\x\;3\; 5\;\x$& $(7,8)\in {\overline S}_{\sigma}\times S_{\sigma}$&(1,1) &(3,2)\\
\hline
\\[-1em]
5&$2\;6\;8\;1\;4\;7\;\x\;9\;10\;\x\;3\; 5\;\x$&
$(9,10)\in S_{\sigma}\times {\overline S}_{\sigma}$&(1,1)&(2,3)\\
\hline
\\[-1em]
6& $2\;6\;8\;1\;4\;7\;\x\;9\;10\;12\;3\; 5\;11\;\x$
&$(11,12)\in {\overline S}_{\sigma}\times  S_{\sigma}$ &(0,0)&(3,2)\\
\hline
\\[-1em]
7&$2\;6\;8\;1\;4\;7\;14\;9\;10\;12\;3\; 5\;11\;\x\;13\;\x$& $(13,14)\in S_{\sigma}\times S_{\sigma}$&(0,1)&(2,2) \\
\hline
\\[-1em]
8&$2\;6\;8\;1\;4\;7\;14\;9\;10\;12\;3\; 5\;11\;15\;16\;13\;\x$& $(15,16)\in {\overline S}_{\sigma}\times  S_{\sigma}$&(1,0)&(2,1)
\end{tabular}
\caption{The restrictions $\pi_j$ of $\sigma\in \X_{16}$ for $j\in [8]$.}
\end{figure}

\medskip

Hence the corresponding step sequence is $(U,U,L_1, D, U,D,L_2, D)$ with height sequence $(0,1,2,2,1,2,1, 1)$. The corresponding  path diagram
 $(\omega,(\xi,\xi'))$ is illustrated in Figure~\ref{F1}.

\begin{lem}\label{bijection}  The mapping
$\Phi: \sigma\mapsto (\omega,(\xi,\xi'))$ is a bijection from $\X_{2n}$ to  $\mathcal{PD}_{n}$ such that
\begin{subequations}
\begin{align}
\dom\,\sigma&=|L_2|_{\omega},\label{3.3a}\\
\des\sigma&=|U|_{\omega}+|L_2|_{\omega},\\
\231\sigma&=\sum_{i=1}^n (\xi_i+\xi_i'),\\
\312\sigma&=|U|_{\omega}+|D|_{\omega}+\sum_{i=1}^n (h(\omega_i)-\xi_i+h(\omega_i)-\xi'_i),\\
(h(\omega_{j-1}), h(\omega_j))&=(s_{\sigma}(2j-1)-1, s_{\sigma}(2j)-1) \qquad j\in [n],\label{3.3e}
\end{align}
\end{subequations}
where $|A|_\omega$ is the number of steps of type $A$ in the path $\omega$.
\end{lem}
\begin{proof}
First of all, Lemma~\ref{lem:signature} ensures that 
$(\omega,(\xi,\xi'))$ is a  path diagram in $\mathcal{PD}_{n}$.
The trivial verification of \eqref{3.3a}-\eqref{3.3e} is omitted. Note that 
$$
0\leq \xi_j\leq s_{\sigma}(2j-1)-1, \quad 0\leq \xi_j'\leq s_{\sigma}(2j)-1
$$
To show  that $\Phi$  is a bijection,  we construct the inverse mapping  $\Phi^{-1}$. Starting from a path diagram $(\omega,(\xi,\xi'))\in\mathcal{PD}_n$, 
by $\omega$ and \eqref{MP}, we determine  the weak signature $S_{\sigma}$.
For $j\in [n]$ we construct the permutation $\sigma$ step by step 
by inserting $2j-1$ and $2j$ in
 the word $\pi_{j-1}$  with $\pi_0=\x$ as follows:
\begin{enumerate}
\item if the $j$-th step is \emph{rise},
we replace  the $\xi_j+1$-th slot from right to left 
in $\pi_{j-1}$ by $\x (2j-1)\x$ and then replace 
the $\xi_j'+1$-th slot by $2j\x$;
\item if the $j$-th step is \emph{fall},
we replace  
the $\xi_j+1$-slot in $\pi_{j-1}$ by $(2j-1)\x$ and then 
the $\xi_j'+2$-th slot in by $(2j)$;
\item if the $j$-th step is \emph{level of type 1},
we replace  the $\xi_j+1$-th slot in $\pi_{j-1}$ by $(2j-1)\x$ and then replace the $\xi_j'+1$-slot by $2j\x$;
\item if the $j$-th step is \emph{level of type 2},
we replace  the $\xi_j+1$-th slot in $\pi_{j-1}$ by $\x(2j-1)\x$ and then replace the $\xi_j'+2$-th slot by $2j$.
\end{enumerate}
The permutation $\sigma$ is obtained by deleting the final slot $\x$ in $\pi_{n}$.
\end{proof}

%
%

\begin{lem} 
For $\sigma\in \X_{2n}$ let 
$\Phi(\sigma)=(\omega,(\xi,\xi'))$ be the path diagram in $\mathcal{PD}_{n}$.
Then
\begin{subequations}
\begin{align}\label{xi}
\lema\,\sigma&=\sum_{i=1}^n\chi(\xi_i'=h(w_i)),\\
\loma\,\sigma&=\sum_{i=1}^n\chi(\xi_i=h(w_i))
\chi(s_i\in \{D, \,  L_1\}),\\
\remi\,\sigma&=\sum_{i=1}^n\chi(\xi_i=0)\chi(s_i\in \{U,\, L_1\}),\\
\romi\,\sigma&=\sum_{i=1}^n\chi(\xi_i=0)\chi(s_i\in \{U,\, D, \,L_1\}),
\end{align}
\end{subequations}
\end{lem}
\begin{proof}
The trivial verification is omitted.
\end{proof}

\begin{proof}[Proof of Theorem~\ref{master-cf}]
For $\sigma\in \X_{2n}$, if 
 $\Phi(\sigma)=(\omega,(\xi,\xi'))\in \mathcal{PD}_n$, we derive from the above lemma 
\begin{align}\label{weight:path}
a^{\lema\,\sigma}\bar a^{\loma\,\sigma}
b^{\romi\,\sigma}{\bar b}^{\remi\,\sigma}
p^{\231\sigma }q^{\312\sigma }y^{\TB\, \sigma}t^{\des \sigma}
=\prod_{i=1}^n w(s_i, (\xi_i, \xi'))
\end{align}
with $h$ being the height of the step $s_i$,
\begin{align}\label{weight:dh}
w(s_i,  (\xi_i, \xi'_i))
=\begin{cases}
a^{\chi(\xi_i'=h+1)}b^{\chi(\xi_i=0)}\,
p^{\xi_i}q^{h-\xi_i}\cdot p^{\xi'_i}q^{h+1-\xi'_i} \cdot t&\; \textrm{if}\quad s_i=U;\\
a^{\chi(\xi_i'=h)}{\bar a}^{\chi(\xi_i=h)}b^{\chi(\xi_i=0)}
\,
p^{\xi_i}q^{h-\xi_i}\cdot p^{\xi'_i}q^{h-1-\xi'_i}&\; \textrm{if}\quad s_i=D;\\
a^{\chi(\xi_i'=h)}{\bar a}^{\chi(\xi_i=h)}b^{\chi(\xi'_i=0)}
\,p^{\xi_i}q^{h-\xi_i}\cdot p^{\xi'_i}q^{h-\xi'_i}&\; \textrm{if}\quad s_i=L_1;\\
a^{\chi(\xi_i'=h)}
\,p^{\xi_i}q^{h-\xi_i}\cdot p^{\xi'_i}q^{h-\xi'_i}\;ty&\; \textrm{if}\quad s_i=L_2.
\end{cases}
\end{align}

Therefore, the corresponding  polynomial and weights become
\begin{align*}
 X_n(a,\bar a,b, \bar b, p,q,y, t)&=\sum_{\sigma\in \mathcal{X}_{2n}}
a^{\lema\,\sigma}\bar a^{\loma\,\sigma}
b^{\romi\,\sigma} {\bar b}^{\remi\,\sigma} p^{\231\sigma }q^{\312\sigma }y^{\TB\, \sigma}t^{\des \sigma}\\
&=\sum_{\omega\in \mathcal{MP}_n}\sum_{(\xi_i, \xi_i')} w(s_i,  (\xi_i, \xi'_i))
\end{align*} 
with
\begin{align}\label{general-weight:Mpath}
\sum_{(\xi_i, \xi_i')} w(s_i,  (\xi_i, \xi'_i))
=\begin{cases}
[b,h+1]_{q,p}\cdot[a,h+2,\bar b]_{p,q}\cdot t&\; \textrm{if}\quad s_i=U;\\
[\bar a,h+2,b]_{p,q}\cdot [a,h+1]_{p,q}&\; \textrm{if}\quad s_i=D;\\
[\bar a,h+1, b]_{p,q}\cdot[a,h+1,\bar b]_{p,q}&\; \textrm{if}\quad s_i=L_1;\\
[a,h+1]_{p,q}\cdot[b,h+1]_{q,p}\;ty&\; \textrm{if}\quad s_i=L_2.
\end{cases}
\end{align}
with  $[x,1,y]_{p,q}=xy$ and 
$$
[x,n,y]_{p,q}=xp^{n-1}+\sum_{i=1}^{n-2}p^{n-1-i}q^i+yq^{n-1}.
$$
By Lemma \ref{flajolet} we derive the corresponding continued fraction for the generating function of $X_n(a,\bar a,b, \bar b, p,q,y, t)$.\end{proof}


\section{Continued fractions for  cycles 
in  D-permutations and E-permutations}
A \emph{surjective pistol} on $[2n]$ is the graph of a  surjective mapping $f: [2n]\to 2[n]$ such that 
$$
f(i)\geq i\quad \textrm{for}\quad i\in [2n].
$$
Let $\mathcal{P}_{2n}$ be the set of surjective pistols on $[2n]$. A surjective pistol is depicted in 
Figure~\ref{alterp}. 
\begin{figure}[t]
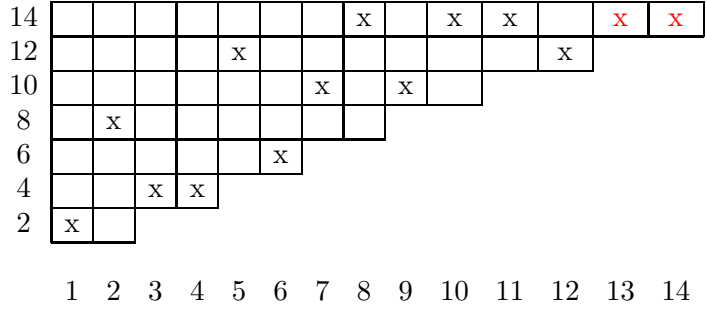
{\small
\begin{tabular}{ccccccccccccccc|}
\cline{2-15} 
14&\multicolumn{1}{|c|}{}&\multicolumn{1}{|c|}{}&\multicolumn{1}{|c|}{}&\multicolumn{1}{|c|}{}&\multicolumn{1}{|c|}{}&\multicolumn{1}{|c|}{}&\multicolumn{1}{|c|}{}&\multicolumn{1}{|c|}{x}&\multicolumn{1}{|c|}{}&\multicolumn{1}{|c|}{x}&\multicolumn{1}{|c|}{x}&\multicolumn{1}{|c|}{}&\multicolumn{1}{|c|}{\red{x}}&\multicolumn{1}{|c|}{\red{x}}\\
\cline{2-15} 
12&\multicolumn{1}{|c|}{}&\multicolumn{1}{|c|}{}&\multicolumn{1}{|c|}{}&\multicolumn{1}{|c|}{}&\multicolumn{1}{|c|}{x}&\multicolumn{1}{|c|}{}&\multicolumn{1}{|c|}{}&\multicolumn{1}{|c|}{}&\multicolumn{1}{|c|}{}&\multicolumn{1}{|c|}{}&\multicolumn{1}{|c|}{}&\multicolumn{1}{|c}{x}&\multicolumn{1}{|c}{}&\multicolumn{1}{c}{}\\
\cline{2-13} 
10&\multicolumn{1}{|c|}{}&\multicolumn{1}{|c|}{}&\multicolumn{1}{|c|}{}&\multicolumn{1}{|c|}{}&\multicolumn{1}{|c|}{}&\multicolumn{1}{|c|}{}&\multicolumn{1}{|c|}{x}&\multicolumn{1}{|c|}{}&\multicolumn{1}{|c|}{x}&\multicolumn{1}{|c|}{}&\multicolumn{1}{c}{}&\multicolumn{1}{c}{}&\multicolumn{1}{c}{}&\multicolumn{1}{c}{}\\
\cline{2-11} 
8&\multicolumn{1}{|c|}{}&\multicolumn{1}{|c|}{x}&\multicolumn{1}{|c|}{}&\multicolumn{1}{|c|}{}&\multicolumn{1}{|c|}{}&\multicolumn{1}{|c|}{}&\multicolumn{1}{|c|}{}&\multicolumn{1}{|c|}{}&\multicolumn{1}{c}{}&\multicolumn{1}{c}{}&\multicolumn{1}{c}{}&\multicolumn{1}{c}{}&\multicolumn{1}{c}{}&\multicolumn{1}{c}{}\\
\cline{2-9} 
6&\multicolumn{1}{|c|}{}&\multicolumn{1}{|c|}{}&\multicolumn{1}{|c|}{}&\multicolumn{1}{|c|}{}&\multicolumn{1}{|c|}{}&\multicolumn{1}{|c|}{x}&\multicolumn{1}{c}{}&\multicolumn{1}{c}{}&\multicolumn{1}{c}{}&\multicolumn{1}{c}{}&\multicolumn{1}{c}{}&\multicolumn{1}{c}{}&\multicolumn{1}{c}{}&\multicolumn{1}{c}{}\\
\cline{2-7} 
4&\multicolumn{1}{|c|}{}&\multicolumn{1}{|c|}{}&\multicolumn{1}{|c|}{x}&\multicolumn{1}{|c|}{x}&\multicolumn{1}{c}{}&\multicolumn{1}{c}{}&\multicolumn{1}{c}{}&\multicolumn{1}{c}{}&\multicolumn{1}{c}{}&\multicolumn{1}{c}{}&\multicolumn{1}{c}{}&\multicolumn{1}{c}{}&\multicolumn{1}{c}{}&\multicolumn{1}{c}{}\\
\cline{2-5} 
2&\multicolumn{1}{|c|}{x}&\multicolumn{1}{|c|}{}&\multicolumn{1}{c}{}&\multicolumn{1}{c}{}&\multicolumn{1}{c}{}&\multicolumn{1}{c}{}&\multicolumn{1}{c}{}&\multicolumn{1}{c}{}&\multicolumn{1}{c}{}&\multicolumn{1}{c}{}&\multicolumn{1}{c}{}&\multicolumn{1}{c}{}&\multicolumn{1}{c}{}&\multicolumn{1}{c}{}\\
\cline{2-3} \\
&\multicolumn{1}{c}{1}&\multicolumn{1}{c}{2}&\multicolumn{1}{c}{3}&\multicolumn{1}{c}{4}&\multicolumn{1}{c}{5}&\multicolumn{1}{c}{6}&\multicolumn{1}{c}{7}&\multicolumn{1}{c}{8}&\multicolumn{1}{c}{9}&\multicolumn{1}{c}{10}&\multicolumn{1}{c}{11}&\multicolumn{1}{c}{12}&\multicolumn{1}{c}{13}&\multicolumn{1}{c}{14}\\
\end{tabular}
\caption{\label{alterp}
 $f=2\,8\,4\,4\,12\,6\,10\,14\,10\,14\,14\,12\,14\,14\in \mathcal{P}_{14}$}
}
\end{figure}
A point $(k, f(k))$ of a pistol $f\in \mathcal{P}_{2n}$  is
\begin{itemize}
\item  {\sl maximum} if  $f(k)=2n$ and $k\leq 2n-2$, 
\item {\sl fixed} if  $f(k)=k$ and $k<2n$;
\item {\sl surfixed} if   $f(k)=k+1<2n$; 
\item {\sl doubled} if $\exists j\neq k$ such that  $f(j)=f(k)$.
\end{itemize}

A maximum $(k, f(k))$ is even (resp. odd) if $k$ is even  (resp. odd). The number of even (resp. odd) 
maxima of $f$ is denoted by $\mp(f)$ (resp. $\mi (f)$). 
The number of doubled (resp. isolated) fixed points is denoted 
by $\fd (f)$ (resp. $\fnd(f)$).  
The number of  doubled (resp. isolated) surfixed points  is denoted by  $\sd(f)$ (resp. $\snd (f)$). 
The six  statistics of the pistol  $f$ in Figure~\ref{alterp} are:
$$
\mi(f)=1,\;\mp(f)=2,\;\fd (f)=2,\;\fnd (f)=1,\;\sd (f)=2, \;\snd (f)=1.
$$
Let 
$$
\Gamma_n(\alpha, \beta, \gamma, \bar \alpha, \bar y, \bar z)=
\sum_{f\in \P_{2n}}\alpha^{\mi (f)}
\beta^{\fd (f)}\gamma^{\snd (f)}{\bar \alpha}^{\mp (f)}{\bar \beta}^{\fnd (f)}{\bar
\gamma}^{\sd (f)}.
$$
Dumont~\cite{Du92} conjectured, and  Randrianarivony~\cite{Ra94} and Zeng~\cite{Ze96} proved the following result.
\begin{lem}[Randrianarivony-Zeng]\label{RZ}  We have 
\begin{align}
\sum_{n\geq 0} \Gamma_{n+1}(\alpha, \bar \alpha)x^n
=\cfrac{1}{1-(\alpha\bar \beta+\beta\bar \gamma+\gamma\bar \alpha)x-\cfrac{(\bar \alpha+\beta)(\bar \beta+z)(\bar \gamma+\alpha)x^2}{\ddots}}
\end{align}
where the coefficients under the (n+1)-th row of fraction is 
\begin{multline}
1-[(\alpha+n)(\bar \beta+n)+(\beta+n)(\bar \gamma+n)+(\gamma+n)(\bar \alpha+n) -n(n+1)]x\\
-\cfrac{(n+1)(\bar \alpha+\beta+n)(\bar \beta+\gamma+n)(\bar \gamma+\alpha+n)x^2}{\ddots}
\end{multline}
\end{lem}
Lazar-Wachs~\cite[Lemma 5.2]{LW19} proved the following result.
\begin{lem}[Lazar-Wachs]\label{lem:LW}
There is a bijection 
$$
\phi: \D_{2n}\mapsto \{f\in \P_{2n+2}: \textrm{$f$ has no even maximal}\}
$$
such that for all $\sigma\in \D_{2n}$ and $j\in [2n]$, the following properties hold:
\begin{enumerate}
\item $j$ is an even cycle maximum of $\sigma$ iff it is a fixed point of $\phi(\sigma)$,
\item $j$ is an even fixed point of $\sigma$ iff it is an isolated fixed point of $\phi(\sigma)$,
\item $j$ is an odd fixed point of $\sigma$ iff it is an odd maximum 
of $\phi(\sigma)$.
\end{enumerate}
\end{lem}\label{lem-xy-cycle}

\begin{proof}[Proof of Theorem~\ref{LW-conj}]
By Lemma~\ref{lem:LW} we obtain immediately 
$$
\sum_{\sigma \in \D_{2n}}x_0^{\f_{\rm e}(\sigma)}x_1^{\f_{\rm o}(\sigma)}z^{\cyc(\sigma)}=
\Gamma_{n+1}(x_1z,z,1,0,x_0z, 1)\quad \textrm{for $n\geq 1$},
$$ 
where $\f_{\rm e}(\sigma)$ (resp. $\f_{\rm o}(\sigma)$)
is  the number of 
even (resp. odd) fixed points of $\sigma$.

It follows from   Lemma~\ref{RZ} that
\begin{subequations}
\begin{align}\label{cf-xy-cycle}
1+\sum_{n\geq 1}\sum_{\sigma \in \D_{2n}}x_0^{\f_{\rm e}(\sigma)}x_1^{\f_{\rm o}(\sigma)}z^{\cyc(\sigma)} x^n 
=\frac{1}{1-(x_0x_1z^2+z)x-\cfrac{z(x_0z+1)(x_1z+1)x^2}{\ddots}} 
\end{align}
where the coefficients under the (n+1)-th row of fraction is 
\begin{multline}
1-[(x_1z+n)(x_0z+n)+(z+n)(1+n)]x\\
-\cfrac{(n+1)(z+n)(x_0z+1+n)(x_1z+1+n)x^2}{\ddots}
\end{multline}
\end{subequations}
Comparing Corollary~\ref{cf-e-cyc} with $t=1$ and \eqref{cf-xy-cycle} with  $x_0=x_1=1$ we see that 
$\sum_{\sigma \in \E_{2n}}z^{\cyc(\sigma)}  =\sum_{\sigma \in \D_{2n}}z^{\cyc(\sigma)}.$
\end{proof}
\medskip

Note that  setting $x_0=x_1=0$ in \eqref{cf-xy-cycle} yields the S-fraction
\begin{align}
1+\sum_{n\geq 1}\sum_{\sigma \in \D^*_{2n}}z^{\cyc(\sigma)} t^n =\cfrac{1}{1-
\cfrac{1\cdot z\, t}{1-
\cfrac{1^2\, t}{1-\cfrac{2\cdot (z+1)\, t}{1-\cfrac{2^2\, t}{\cdots}}}}}
\end{align}
where $D_{2n}^*$ is the set of  derangements in $\D_{2n}$.

\section{Gamma-decomposition and group actions}

We  give another proof of \eqref{formula:gamma} using \emph{Inter-hopping} action on $\X_{2n}$,  see \cite{EFLL21, BJS10}. 
For  $\sigma\in\X_{2n}$ and  $j\in [n]$, a  doubleton $\{2j-1,2j\}$ is called \emph{free} (in $\sigma$) if both $2j-1$ and $2j$ are in $S_{\sigma}$, or neither  $2j-1$ nor  $2j$ is in $S_{\sigma}$, namely
$(2j-1,2j)\in (S_\sigma\times S_\sigma)\cup (\overline{S}_\sigma\times \overline{S}_\sigma)$.

Let  $w=w_1\cdots w_{2n}\in\X_{2n}$ with a free pair $\{w_i,w_j\}=\{2r-1,2r\}$ for some $i<j$. We define 
the action $\phi_r$ on $w$ such that  $|\des w-\des\phi_r(w)|=1$.
\vskip 0.3cm
{\bf Inter-hopping action\, $\phi_r$}
\begin{itemize}
\item[(A1)]
$(2j-1,2j)\in  \overline{S}_\sigma\times \overline{S}_\sigma$. If the elements $2r-1$ and $2r$ are adjacent in $w$ then $\phi_r(w)$ is obtained by switching $2r-1$ and $2r$. Otherwise, we factorize $w$ as
$$
w=\cdots\beta_0w_i\alpha_1\beta_1\alpha_2\beta_2\cdots\alpha_d\beta_dw_j\alpha_{d+1}\cdots,
$$
where $\alpha_j$ (resp. $\beta_j$) is a maximal sequence of consecutive entries greater than $2r$ (resp. less than $2r-1$). Note that 
neither $\alpha_j$ nor  $\beta_j$ is empty for $1\leq j\leq d$, but
 $\beta_0$ and $\alpha_{d+1}$ are possibly empty.
\begin{itemize}
\item 
If $w_i=2r$ and $w_j=2r-1$, then 
$$
\phi_r(w):=\cdots\beta_0\alpha_1(2r-1)\alpha_2\beta_1\alpha_3\beta_2\cdots\alpha_d\beta_{d-1}(2r)\beta_d\alpha_{d+1}\cdots.
$$
\item
If $w_i=2r-1$ and $w_j=2r$, then 
$$
\phi_r(w):=\cdots\beta_0(2r)\beta_1\alpha_1\beta_2\alpha_2\cdots\beta_d\alpha_d(2r-1)\alpha_{d+1}\cdots.
$$
\end{itemize}
\item[(A2)]
$(2j-1,2j)\in S_\sigma\times S_\sigma$. Then $\phi_r(w)$ is obtained by reversing the process of (A1).
\end{itemize}
For $\sigma\in\X_{2n}$, let $\Orb(\sigma)=\{g(\sigma):g\in\Z_2^n\}$ be the orbit of $\sigma$ under the  Inter-hopping action $\phi$. Let $\tilde{\sigma}$ be the unique $\mathcal{E}$-permutation in $\Orb(\sigma)$. We have the following lemma.
\begin{lem}\label{kl2}
For any $\tilde{\sigma}\in\bar{\X}_{2n}$, we have
\begin{equation}\label{eq-orbite}
\sum_{\sigma\in \Orb(\tilde{\sigma})}p^{\312\sigma}q^{\231\sigma}t^{\des\sigma}y^{\TB(\sigma)}=p^{\312\tilde{\sigma}}q^{\231\tilde{\sigma}}t^{\des\tilde{\sigma}}(1+yt)^{n-2\des\tilde{\sigma}}.
\end{equation}
\end{lem}
\begin{proof}
For $w\in\X_{2n}$ with a free pair $\{2r-1,2r\}$, if $2r-1$ and $2r$ are adjacent in $w$, then obviously, 
$$
\231\phi_r(w)=\231w,\quad \312\phi_r(w)=\312w.
$$
Otherwise, we factorize $w$ as 
$$
w=\cdots\beta_0w_i\alpha_1\beta_1\alpha_2\beta_2\cdots\alpha_d\beta_dw_j\alpha_{d+1}\cdots.
$$
We assume that $w_i=2r$ and $w_j=2r-1$ with $i<j$ (the  $i>j$ case  is similarly). Then
$$
\phi_r(w):=\cdots\beta_0\alpha_1(2r-1)\alpha_2\beta_1\alpha_3\beta_2\cdots\alpha_d\beta_{d-1}(2r)\beta_d\alpha_{d+1}\cdots.
$$
We see that the Inter-hopping action $\phi_r$ does not change the relative orders of $\alpha_i$ and $\beta_j$ respectively, so
\begin{itemize}
\item [(i)]
if there is a descent pair in $\alpha_i$ ($\beta_j$, respectively), 
then this descent pair can only form $\231$ or $\312$ patterns with a letter in some $\alpha_k$ ($\beta_l$, respectively);
\item [(ii)]
in $w$, for $a\in\alpha_i$,  $a$ forms a $\231$ pattern with some descent pair say $(\L(\alpha_j),\F(\beta_j))$, if and only if  in $\phi_r(w)$, $a$ forms a $\231$ pattern either with descent $$(\L(\alpha_j),\F(\beta_{j-1}))\quad {\rm  or}\quad (\L(\alpha_j),\F(\beta_j)).
$$
 For $b\in\beta_i$,  $b$ formed a $\231$ pattern with some descent pairs $(\L(\alpha_j),\F(\beta_j))$ if and only if $b$ forms a $\231$ pattern either with descent pair $(\L(\alpha_{j-1}),\F(\beta_j))$ or with descent pair $(\L(\alpha_{j}),\F(\beta_j))$;
\item[(iii)]
the number of $\231$ patterns formed by $2r-1$ and $2r$ with some descent pairs  in $w$ is  equal to the number of $\231$ patterns formed
 by $2r-1$ and $2r$ with some descent pairs  in $\phi_r(w)$.
\end{itemize}
Combining (i), (ii) and (iii), we obtain $\231w=\231\phi_r(w)$.
Similarly, we have $\312w=\312\phi_r(w)$.
Setting each free pair of $w$ not in $S_w$ yields the unique 
$\mathcal{E}$-permutation,  say $\tilde{\sigma}$,  in $\Orb(\sigma)$. Moreover, if $(2j, 2j-2k-1)$ is a descent pair, then the doubletons 
$\{2j, 2j-1\}$ and  $\{2j-2k, 2j-2k-1\}$ are not free. So,
the number of  free doubletons in $\tilde{\sigma}$ is
$n-2\des(\tilde{\sigma})$, which implies
 the identity \eqref{eq-orbite}.
\end{proof}
For convenience, we call an element \emph{changed its type}, if it changed from an even left-to-right maxima (resp. odd right-to-left minima) to a non even left-to-right maxima (resp. a non odd right-to-left minima), or from a non even left-to-right maxima to an even left-to-right maxima (resp. an odd right-to-left minima).
\vskip 0.3cm
$\mathbf{Involution\, \theta_r}$
\vskip 0.2cm
For $\sigma\in\X_{2n}$, and $\{2r-1,2r\}$ is free 
with $(2r-1,2r)\in  \overline{S}_\sigma\times \overline{S}_\sigma$. Then $\theta_r(\sigma)=\tau$, where $\tau\in\X_{2n}$ is obtained by exchanging $2r-1$ and $2r$ in $\sigma$.
\begin{lem}\label{kl1}
 For $w\in\X_{2n}$, if $\{2r-1,2r\}$ is free 
with $(2r-1,2r)\in  \overline{S}_\sigma\times \overline{S}_\sigma$, then 
 \begin{align}\label{N1}
 \lema\,\theta_r(w)&=\lema\,\phi_r(w)\\\label{N2}
 \romi\,\theta_r(w)&=\romi\,\phi_r(w),
 \end{align}
 and 
 \begin{align}\label{K1}
 \lema (\phi_r\circ\theta_r)(w)&=\lema\, w\\\label{K2}
 \romi (\phi_r\circ\theta_r)(w)&=\romi\, w.
 \end{align}
 If $\{2r-1,2r\}$ is a free pair with 
 $(2r-1,2r)\in S_\sigma\times S_\sigma$, then 
 \begin{align}\label{M1}
 \lema\,(\theta_r\circ\phi_r)(w)=\lema\, w\\\label{M2}
 \romi\,(\theta_r\circ\phi_r)(w)=\romi\, w.
 \end{align}
\end{lem}
\begin{proof}
We take the factorization of $w$ as in (A1), i.e., 
\begin{align}\label{J1}
w=\cdots\beta_0w_i\alpha_1\beta_1\alpha_2\beta_2\cdots\alpha_d\beta_d w_j\alpha_{d+1}\cdots,
\end{align}
 here we assume $w_i=2r-1$, $w_j=2r$.
Then,
\begin{align}\label{L1}
\phi_r(w)&=\cdots\beta_0(2r)\beta_1\alpha_1\beta_2\alpha_2\cdots\beta_d\alpha_d(2r-1)\alpha_{d+1}\cdots\\\label{L2}
\theta_r(w)&=\cdots\beta_0(2r)\alpha_1\beta_1\alpha_2\beta_2\cdots\alpha_d\beta_d(2r-1)\alpha_{d+1}\cdots.
\end{align}
Since all the letters in $\beta_i$ are smaller than $2r-1$ and all the letters in $\alpha_j$ are larger than $2r$. So, the even left-to-right maxima element, if exists, it must be in one of $\alpha_i$ or the element before $\beta_0$, or it is $2r$.
And the odd right-to-left minima element, if exists, it must be in one of $\beta_j$ or the element after $\alpha_{d+1}$ or it is $2r-1$.
Compare~\eqref{L1} with \eqref{J1} and~\eqref{L2} with \eqref{J1}, we can see that under the actions $\theta_r$ and $\phi_r$, if some elements changed their type, they must in $\{2r-1, 2r\}$. And compare~\eqref{L1} and~\eqref{L2}, we have~\eqref{N1} and~\eqref{N2}. By the same reason, we can derive~\eqref{K1},\eqref{K2},\eqref{M1} and \eqref{M2}.
\end{proof}
Note that for $\sigma\in\X_{2n}$, neither $\theta_r$ and $\phi_r$ changes the number of $\312$ or $\231$ patterns when them action on $\sigma$. Then combine Lemma~\ref{kl2} and Lemma~\ref{kl1}, we can finally derive Theorem~\ref{thm:gamma}.
\section{Factorization of $\gamma$-coefficients and group actions}
We  prove Theorem~\ref{thm:main-pq}.  Clearly 
 it suffices to prove 
\begin{align}\label{eq:gamma-h}
\sum_{\sigma\in\overline{\X}_{2n,k}}p^{\231\sigma}q^{\312\sigma}=(p+q)^{2k}\sum_{\sigma\in\widehat{\X}_{2n,k}}q^{\312\sigma-k}p^{\231\sigma-k}
\end{align}
with 
$\overbar{\X}_{2n,k}=\{\sigma\in \overbar{\X}_{2n}: \des\,\sigma=k\}$ and $\widehat{\X}_{2n,k}=\{\sigma\in \widehat{\X}_{2n}: \des\,\sigma=k\}$. We define  an  action $\varphi_x$ on $\overbar{\X}_{2n}$ as the following.

For $\sigma\in \overbar{\X}_{2n}$, and $x\in S_{\sigma}$,  we define $\varphi_x(\sigma)$ 
as follows:

\begin{itemize}
\item[(A)]
If $x=2i-1$ and $s_{\sigma}(2i-1)$ is odd, we factorize $\sigma=\tau_3'\tau_2'\tau_1'(2i)\tau_1\tau_2\tau_3$, where $\tau_1$ and $\tau_2'$ is the maximal sequence of consecutive entries greater than $2i$, $\tau_2$ and $\tau_1'$ is the maximal sequence of consecutive entries smaller than $2i$. We define
$$
\varphi_x(\sigma)=\begin{cases}
\tau_3'\tau_2'\tau_1'\tau_1\tau_2(2i)\tau_3&\textrm{if $l_{\sigma}(2i)$ is even};\\
\tau_3'(2i)\tau_2'\tau_1'\tau_1\tau_2\tau_3&\textrm{if $l_{\sigma}(2i)$ is odd}.
\end{cases}
$$
\item[(B)]
If $x=2j$, and $s_{\sigma}(2j)$ is odd, we factorize $\sigma=\tau_3'\tau_2'\tau_1'(2j-1)\tau_1\tau_2\tau_3$, where $\tau_1$ and $\tau_2'$ is the maximal sequence of consecutive entries greater than $2j-1$, $\tau_2$ and $\tau_1'$ is the maximal sequence of consecutive entries smaller than $2j-1$. Define
$$
\varphi_x(\sigma)=\begin{cases}
\tau_3'(2j-1)\tau_2'\tau_1'\tau_1\tau_2\tau_3&\textrm{if $r_{\sigma}(2j-1)$ is even};\\
\tau_3'\tau_2'\tau_1'\tau_1\tau_2(2j-1)\tau_3&\textrm{if $r_{\sigma}(2j-1)$ is odd}.
\end{cases}
$$
\item[(C)]
If $x=2i-1$ and $s_{\sigma}(2i-1)$ is even,
\begin{itemize}
\item [(i)]
if $l_{\sigma}(2i-1)=2a$ for some $a\in\N$, then
$$
\varphi_x(\sigma)=w\in\overbar{\X}_{2n},
$$
where $w$ is constructed as follows.
\\
Let $w_0$ be the subsequence of $\sigma$ consisting of those elements which are smaller than $2i-1$ in $S_{\sigma}$. And $w_1$ ($\sigma_1$, respectively) is the subsequence of $w$ ($\sigma$, respectively) which consists of all the elements in $S_{\sigma}$.
We obtain $w_{01}$ by inserting $2i-1$ in $w_0$ such that the number of elements at the left side of $2i-1$ in $w_{01}$ minus two times of the number of even-odd descents at the left side of $2i-1$ in $w_{01}$ is $2a+1$. Then we insert the rest elements of $S_{\sigma}$ in $w_{01}$ one by one in increasing order to obtain $w_1$, such that the order of the rest elements in $S_{\sigma}$ appear in $w_1$ is the same as in $\sigma_1$. (The same order means when restrict $w_1$ and $\sigma_1$ on the elements $\{1,2,\ldots,k\}$ for $|S_{\sigma}|\geq k\geq 2i$ then the elements at left side of $k$ minus two times of the number of even-odd descents are the same in the restrictions of $w_1$ and $\sigma_1$) Last, we insert the elements of $[2n]\setminus S_{\sigma}$ one by one in increasing order in $w_1$ to obtain $w$ such that $(l_w(i),r_w(i))=(l_{\sigma}(i),r_{\sigma}(i))$.
\item[(ii)]
if $l_{\sigma}(2i-1)=2a+1$ for some $a\in\N$, then
$$
\varphi_x(\sigma)=w\in\overbar{\X}_{2n},
$$
where $w$ is constructed by following almost the process in (i), except when obtain $w_{01}$ by inserting $2i-1$ in $w_0$ such that the number of the elements at the left side of $2i-1$ in $w_{01}$ minus two times the number of even-odd descents at the left side of $2i-1$ in $w_{01}$ is $2a$.
\end{itemize}
\item[(D)]
If $x=2j$ and $s_{\sigma}(2j)$ is even,
\begin{itemize}
\item[(i)]
if $r_{\sigma}(2j)=2a$ for some $a\in\N$, then
$$
\varphi_x(\sigma)=w\in\overbar{\X}_{2n},
$$
Let $w_0$ be the subsequence of $\sigma$ which consists of the elements in $S_{\sigma}$ and smaller than $2j$. And $w_1$ ($\sigma_1$, respectively) be the subsequence of $w$ ($\sigma$, respectively) which consists of all the elements in $S_{\sigma}$.
We obtain $w_{01}$ by inserting $2j$ in $w_0$ such that the number of elements at the right side of $2j$ in $w_{01}$ minus two times of the number of even-odd descents at the right side of $2j$ in $w_{01}$ is $2a+1$. Then we insert the rest elements of $S_{\sigma}$ in $w_{01}$ one by one in increasing order to obtain $w_1$, such that the order of the rest elements in $S_{\sigma}$ appear in $w_1$ is the same as in $\sigma_1$. Last, we insert the elements of $[2n]\setminus S_{\sigma}$ one by one in increasing order in $w_1$ to obtain $w$ such that $(l_w(i),r_w(i))=(l_{\sigma}(i),r_{\sigma}(i))$.
\item[(ii)]
if $r_{\sigma}(2j)=2a+1$ for some $a\in\N$, then 
$$
\varphi_x(\sigma)=w\in\overbar{\X}_{2n},
$$
where $w$ is constructed by following almost the process in (i), except when obtain $w_{01}$ by inserting $2j$ in $w_0$ such that the number of the elements at the right side of $2j$ in $w_{01}$ minus two times of the number of even-odd descents at the right side of $2j$ in $w_{01}$ is $2a$.
\end{itemize}
\end{itemize}

\begin{exam}
For $\sigma=56347812910\in\X_{10}$,
the set of descent tops and descent bottoms is $S_{\sigma}=\{1,3,6,8\}$. So we have,
$\varphi_1(\sigma)=25634781910$, $\varphi_3(\sigma)=56124783910$, $\varphi_6(\sigma)=57834612910$, $\varphi_8(\sigma)=56348127910$. We illustrate the $\varphi$ act on $1$ and $8$ in Figure~\ref{V}.
\end{exam}
\begin{figure}[h]
\begin{tikzpicture}[scale=0.2]
\draw(0,0)--(6,6);
\fill(0,0) circle (7pt);
\node[below] at (0,0) {$0$};
\fill(5,5) circle (7pt);
\node[below] at (5,5) {$5$};
\fill(6,6) circle (7pt);
\node[above] at (6,6) {$6$};
\draw(6,6)--(9,3);
\fill(9,3) circle (7pt);
\node[below] at (9,3) {$3$};
\draw(9,3)--(10,4);
\fill(10,4) circle (7pt);
\node[above] at (10,4) {$4$};
\draw(10,4)--(14,8);
\fill(13,7) circle (7pt);
\node[below] at (13,7) {$7$};
\fill(14,8) circle (7pt);
\node[above] at (14,8) {$8$};
\draw(14,8)--(21,1);
\fill(21,1) circle (7pt);
\node[below] at (21,1) {$1$};
\draw(21,1)--(22,2);
\fill(22,2) circle (7pt);
\node[above] at (22,2) {$2$};
\draw(22,2)--(31,10);
\fill(31,10) circle (7pt);
\node[below] at (30,9) {$9$};
\fill(30,9) circle(7pt);
\node[above] at (31,10) {$10$};
\draw[>=latex,->,dashed](22,2)--(2,2);
\draw(35,0)--(41,6);
\fill(35,0) circle (7pt);
\node[below] at (35,0) {$0$};
\fill(40,5) circle (7pt);
\node[below] at (40,5) {$5$};
\fill(41,6) circle (7pt);
\node[above] at (41,6) {$6$};
\draw(41,6)--(44,3);
\fill(44,3) circle (7pt);
\node[below] at (44,3) {$3$};
\draw(44,3)--(45,4);
\fill(45,4) circle (7pt);
\node[above] at (45,4) {$4$};
\draw(45,4)--(49,8);
\fill(48,7) circle (7pt);
\node[below] at (48,7) {$7$};
\fill(49,8) circle (7pt);
\node[above] at (49,8) {$8$};
\draw(49,8)--(56,1);
\fill(56,1) circle (7pt);
\node[below] at (56,1) {$1$};
\draw(56,1)--(57,2);
\fill(57,2) circle (7pt);
\node[above] at (57,2) {$2$};
\draw(57,2)--(66,10);
\fill(66,10) circle (7pt);
\node[below] at (65,9) {$9$};
\fill(65,9) circle(7pt);
\node[above] at (66,10) {$10$};
\draw[>=latex,->,dashed](48,7)--(63,7);
\end{tikzpicture}
\caption{$\varphi_1(\sigma)=25634781910$ and $\varphi_8(\sigma)=56348127910$}\label{V}
\end{figure}
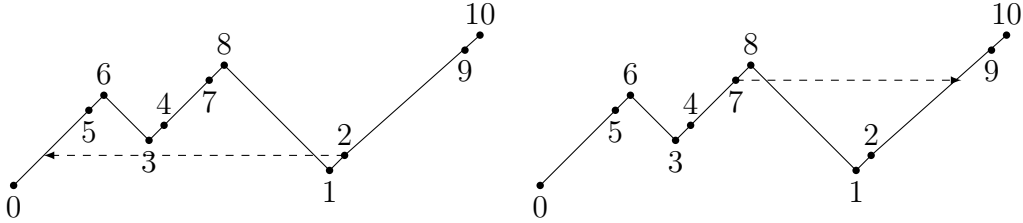

We extend $\varphi_{x}$ to all $x\in [2n]$ by 
$\varphi_{x}(\sigma)=\sigma$ for $x\notin S_{\sigma}$. Clearly 
the action $\varphi_x$ is an involution on $\overbar{\X}_{2n}$, besides two actions 
 $\varphi_x$ and $\varphi_y$ commute for all $x,y\in [2n]$. Hence for any subset $X\subset[2n]$ we may define the action $\varphi_X$ on $\sigma\in\overbar{\X}_{2n}$ by
\begin{equation}\label{def:action}
\varphi_X(\sigma)=\prod_{x\in X}\varphi _x(\sigma).
\end{equation}

\begin{lem}\label{lem3}
Let $\sigma\in\overbar{\X}_{2n}$ with weak signature $S_{\sigma}$. 
For any $x\in S_{\sigma}$,
let
$$
\Delta_{\varphi_x}(\sigma)=(\231\sigma',\312\sigma')
-(\231\sigma,\312\sigma)
$$ 
with $\sigma':=\varphi_x(\sigma)$, then  $S_\sigma=S_{\sigma'}$ and 
\begin{enumerate}[(a)]
\item
If $x$ and $s_{\sigma}(x)$ are odd, then
\begin{align*}
\Delta_{\varphi_x}(\sigma) 
=\begin{cases}
  (-1, 1)& \textrm{if $l_{\sigma}(x+1)$  is even}
   \\
   (1, -1)&\textrm{if  $l_{\sigma}(x+1)$ is odd}; 
\end{cases}
\end{align*}
\item
If $x$ is even but $s_{\sigma}(x)$ is odd, then
\begin{align*}
\Delta_{\varphi_x}(\sigma) 
=\begin{cases}
  (1, -1) & \textrm{if $l_{\sigma}(x-1)$  is even}
  \\
   (-1, 1) &\textrm{if  $l_{\sigma}(x-1)$ is odd;} 
\end{cases}
\end{align*}
\item
If $x$ is odd but $s_{\sigma}(x)$ is even, then
\begin{align*}
\Delta_{\varphi_x}(\sigma) 
=\begin{cases}
  (-1,1)& \textrm{if $l_{\sigma}(x)$  is even}
   \\
   (1, -1)& \textrm{if  $l_{\sigma}(x)$ is odd;}
 \end{cases}  
\end{align*}
\item
If $x$  and $s_{\sigma}(x)$ are even, then
\begin{align*}
\Delta_{\varphi_x}(\sigma) 
=\begin{cases}
  (1,-1)& \textrm{if $r_{\sigma}(x)$  is even}
   \\
   (-1, 1)&\textrm{if  $r_{\sigma}(x)$ is odd. }
\end{cases}
\end{align*}
\end{enumerate}
\end{lem}
\begin{proof}
For case (a), as the action $\varphi_x$ moves $2i$
to the slot  between two adjacent descent pairs, so the number of $\231$ patterns reduces by $1$ and the number of $\312$ patterns increases by $1$ when $2i$ moves to the slot  between a  descent pair at right. And the number of $\231$ patterns increase by $1$ and the number of $\312$ patterns decreases by $1$ when $2i$ moves to the slot between a descent pair at  left. Case (b) is is similar to case (a).

For case (c), it corresponds to (C) of the action $\varphi_x$. In this case, it is easy to see that 
\begin{itemize}
\item if $l_{\sigma}(x)$ is even,
the construction of $\sigma'=\varphi_x(\sigma)$ decreases the right embracing number of $x$  by 1 and increases the left embracing number of $x$  by 1;
\item if $l_{\sigma}(x)$ is odd,
it increases the right embracing number of $x$  by 1 and  decreases the left embracing number of $x$ by 1. 
\end{itemize}
Case (d) is similar to case (c).
\end{proof}

\begin{lem}\label{lem4} For each $\hat{\sigma}\in\widehat{\X}_{2n,k}$
let $\Orb(\hat{\sigma})$ be the orbit of $\hat{\sigma}$ under $\varphi_S$. Then 
\begin{align}\label{pq-orbite}
\sum_{\sigma\in\Orb(\hat{\sigma})}p^{\231\sigma}q^{\312\sigma}=(p+q)^{2k}\cdot q^{\312\hat{\sigma}-k}p^{\231\hat{\sigma}-k}.
\end{align}
\end{lem}
\begin{proof}
We rewrite \eqref{pq-orbite} as
\begin{align}\label{Lemma_3}
\sum_{\sigma\in\Orb(\hat{\sigma})}p^{\231\sigma}q^{\312\sigma}=
(1+p^{-1}q)^{k} \cdot (1+pq^{-1})^{k}\cdot q^{\312\hat{\sigma}}p^{\231\hat{\sigma}}.
\end{align} 
The fours cases (a)-(3.4a), (b)-(3.4c), (c)-(3.4e) and  (d)-(3.4g)
correspond to the four cases of normalized 
$\mathcal{E}$-permutations in Definition~\ref{np}. If 
$\hat{\sigma}\in\widehat{\X}_{2n}$, 
by Lemma~\ref{lem3}, for any $X\subset S_{\hat{\sigma}}$, applying
action   $\varphi_X$  on a descent top (resp. bottom) of $\hat{\sigma}$   decreases  the number of $\312$ (resp. $\231$) patterns  by 1
and increases the number of $\231$ (resp. $\312$) patterns by 1, namely,  
it modifies the weight of $\hat\sigma$ by  the factor  $pq^{-1}$ (resp. $qp^{-1})$). Since $\hat{\sigma}$ has  exactly $k$ descent tops (resp. bottoms), we obtain \eqref{Lemma_3}.
\end{proof}
\begin{exam}
If  $\sigma=124635\in\widehat{\X}_{6,1}$ with $S=\{3, 6\}$, then
$\varphi_3(\sigma)=126345$, $\varphi_6(\sigma)=124563$ and $\varphi_S(\sigma)=125634$. Thus 
$$
\Orb(\sigma)=\{124635, 126345, 124563, 125634\}.
$$
\end{exam}
By Lemma~\ref{lem4}, summing over all the orbits yields \eqref{eq:gamma-h}.
\qed

\section{Other interpretations of 
Genocchi numbers}

For $\sigma=\sigma_1\sigma_2\ldots\sigma_n\in\S_n$, the statistics $\res\sigma$ is the number of pairs $(i,j)$ such that $1\leq i<j\leq n-1$ and $\sigma_{j+1}>\sigma_{i}>\sigma_{j}$.\footnote{
The $\res$ statistic is called  \emph{artificial statistic} and denoted by ``art'' in~\cite{EFLL21}.}
\begin{defi}\label{peak-valley}
For  $\sigma=\sigma_1\ldots\sigma_n\in \S_n$  with $\sigma_0=\sigma_{n+1}=0$, the value $\sigma_i$ is
\begin{itemize}
\item a peak if $\sigma_{i-1}<\sigma_i$ and $\sigma_i>\sigma_{i+1}$;
\item a valley if $\sigma_{i-1}>\sigma_i$ and $\sigma_i<\sigma_{i+1}$;
\item a double ascent if $\sigma_{i-1}<\sigma_i<\sigma_{i+1}$;
\item a double descent if $\sigma_{i-1}>\sigma_i>\sigma_{i+1}$.
\end{itemize}
\end{defi}
Let $\dd\,\sigma$ be the number of double descents 
in $\sigma$.
Define the enumerative polynomials
\begin{align}
Y_n(a,p,q,y,t)=\sum_{\sigma\in\Y_{2n+1}^*}a^{\lema\,\sigma}p^{\res\,\sigma}q^{\312\,\sigma}y^{\dd\,\sigma}t^{\des\,\sigma}
\end{align}

\begin{thm}\label{oddodd:cf} For $n\geq 1$ we have 
\begin{align}
Y_{n}(a,p,q,y,t)=X_n(a,1, 1,1,  p,q, y,t).
\end{align}
\end{thm}
\begin{proof} It suffices to show  that the two polynomials have the same ordinary generating functions.
In view of Corollary~\ref{cor-master-cf} we need only to prove that 
\begin{align}\label{y-cf}
\sum_{n=0}^{\infty}Y_{n}(a,p,q,y,t)x^{n}=\cfrac{1}{
1  -(1+yt)[a,1]_{p,q}\, x-\cfrac{t[a,1]_{p,q} [a,2]_{p,q}[2]_{p,q}\cdot  x^2}{
\cdots}}
\end{align}
with coefficents 
\begin{align*}\begin{cases}
b_{n-1}&=(1+ty)[a,n]_{p,q}[n]_{p,q} \quad \textrm{for}\quad n\geq 1,\\
\lambda_{n}&=t[a,n]_{p,q}[a,n+1]_{p,q}[n]_{p,q}[n+1]_{p,q}.
\end{cases}
\end{align*}
As the proof is similar to that of Theorem~\ref{master-cf}, we just indicate  the mapping $\psi: \Y_{2n+1}^*\to \mathcal{PD}_n$ with the associated  weight and  omit the details.

Let  $\sigma\in\Y_{2n+1}^*$ and 
$\psi(\sigma)=(w,(\xi,\xi'))$.
For $j\in[n]$ we define  the steps $s_j=w_j-w_{j-1}$ of 
the path $w=(w_0,w_1,\ldots, w_n)$ as follows:
\begin{equation}
s_j=\begin{cases}
U& \textrm{if}\quad 2j-1\quad\textrm{is a valley of} \; \sigma;\\
D&\textrm{if}\quad 2j-1\quad\textrm{is a peak of } \; \sigma;\\
L_1&\textrm{if}\quad 2j-1\quad\textrm{is a double ascent of } \; \sigma;\\
L_2&\textrm{if}\quad 2j-1\quad\textrm{is a double descent of } \; \sigma,
\end{cases}
\end{equation}
and let  $(\xi_j,\xi'_j)=(l_{\sigma}(2j-1),l_{\sigma}(2j))$.
Thus,  the weight of the path diagram $(w,(\xi,\xi'))=\prod_{i=1}^nw(w_i,(\xi_i,\xi'_i))$ is defined by
\begin{align}
w(\omega_i,  (\xi_i, \xi'_i))=\begin{cases}
a^{\chi(\xi'_i=0)}\,q^{\xi_i}p^{h-\xi_i}\cdot q^{\xi'_i}p^{h+1-\xi'_i} \cdot t&\; \textrm{if}\quad \omega_i=U;\\
a^{\chi(\xi'_i=0)}\,q^{\xi_i}p^{h+1-\xi_i}\cdot q^{\xi'_i}p^{h-\xi'_i}&\; \textrm{if}\quad \omega_i=D;\\
a^{\chi(\xi'_i=0)}\,q^{\xi_i}p^{h-\xi_i}\cdot q^{\xi'_i}p^{h-\xi'_i}&\; \textrm{if}\quad \omega_i=L_1;\\
a^{\chi(\xi'_i=0)}\,q^{\xi_i}p^{h-\xi_i}\cdot q^{\xi'_i}p^{h-\xi'_i}\cdot yt&\; \textrm{if}\quad \omega_i=L_2.
\end{cases}
\end{align}
The J-fraction  \eqref{y-cf} follows then from Lemma~\ref{flajolet}.
\end{proof}

A permutation $\sigma\in\Y^*_{2n+1}$  is called  a \emph{$\mathcal{F}$-permutation} if $\sigma$ contains no double descent, and the last entry of $\sigma$ is a peak, i.e., $\sigma_{2n}<\sigma_{2n+1}$. Denote $\overline{\Y}^*_{2n+1}$ as the set of $\mathcal{F}$-permutations with length $\S_{2n+1}$, and $\overline{\Y}^*_{2n+1,k}$  the subset of permutations in $\overline{\Y}^*_{2n+1}$ with $k$ descents.
\\

\begin{thm}\label{oddodd:gamma}
We have
\begin{align}\label{gamma:odd-odd}
Y_n(a,p,q,y,t)=\sum_{k=0}^{\lfloor n/2\rfloor}\gamma_{n,k}(a,p,q)t^k(1+yt)^{n-2k},
\end{align}
whith coefficients 
$$\gamma_{n,k}(a,p,q)=\sum_{\sigma\in\overline{\Y}^*_{2n+1,k}}a^{\lema\,\sigma}p^{\res\sigma}q^{\312\sigma}.
$$
\end{thm}
\begin{proof} Note that $Y_n(a,p,q,0,t)
=\sum_{k=0}^{\lfloor n/2\rfloor} 
\gamma_{n,k}(a,p,q) t^k$ and 
\begin{align}\label{y0-cf}
\sum_{n=0}^{\infty}Y_{n}(a,p,q,0,t)x^{n}=\cfrac{1}{
1  -[a,1]_{p,q}\, x-\cfrac{t[a,1]_{p,q} [a,2]_{p,q}[2]_{p,q}\cdot  x^2}{
\cdots}}
\end{align}
with coefficents 
\begin{align*}\begin{cases}
b_{n-1}&=[a,n]_{p,q}[n]_{p,q} \quad \textrm{for}\quad n\geq 1,\\
\lambda_{n}&=t[a,n]_{p,q}[a,n+1]_{p,q}[n]_{p,q}[n+1]_{p,q}.
\end{cases}
\end{align*}
We derive  \eqref{gamma:odd-odd} by comparing \eqref{y-cf} and \eqref{y0-cf}.
\end{proof}
\begin{remark} We can also prove the above theorem by 
applying Brändén's modified Foata-Strehl action (a.k.a. MFS-action) on $\Y_{2n+1}^*$~\cite{Br08}. 
More precisely,  we move every double descent to a double  ascent position. It is known that that MFS-action does not change the number of $\res$ patterns and the number of $\312$ patterns. It is also clear that MFS-action does not change the even left-to-right maxima. 
\end{remark}

For $\sigma\in\overline{\Y}_{2n+1}^*$ and $i\in [n]$,
the doubleton  $\{2i-1, 2i\}$ is called a \emph{VOP pair} of $\sigma$ if $2i-1$ is either a valley or  peak of $\sigma$. As for $\mathcal{E}$-permutations, we associate  a sequence $(a(1), a(2), \cdots, a(2n))$ with $\sigma$  by 
$$
a(j)=f(j)-g(j)+1 \quad \textrm{for}\quad j\in [2n].
$$ 
where $f(i)$ (resp. $g(i)$) is the number of valleys (resp. peaks) less than $i$ in $\sigma$.

\medskip
\textbf{Fact }. If $\{2i-1, 2i\}$ is a VOP pair of $\sigma\in\overline{\Y}_{2n+1}^*$, then there is one and only one even element in $\{a(2i-1), a(2i)\}$.
\\
For convenience, we define the embracing number of 
 a VOP pair $x=\{2i-1,2i\}$ of $\sigma$ by
  $$
  l_{\sigma}(x)=\begin{cases}
  \231(2i-1)&\textrm{if $a(2i-1)$ is even};\\
  \312(2i)&\textrm{if $a(2i-1)$ is odd}.
  \end{cases}
  $$ 
\begin{defi}
A $\mathcal{F}$-permutation $\sigma\in\overline{\Y}^*_{2n+1}$ is called normalized, if for any VOP pair $x=\{2i-1, 2i\}$ of $\sigma$,  the embracing number  $l_{\sigma}(x)$ is even. 
Let  $\widehat{\Y}^*_{2n+1}$ be the set of normalized $\mathcal{F}$-permutations in $\overline{\Y}^*_{2n+1}$ and  $\widehat{\Y}^*_{2n+1,k}$ the subset of
normalized $\mathcal{F}$-permutations in $\widehat{\Y}^*_{2n+1}$
with $k$ descents by.
\end{defi}

\vskip 0.5cm
\textbf{Action $\bar{\varphi}_x$} on $\overline{\Y}_{2n+1}^*$. For a VOP pair $x=\{2i-1,2i\}$ of $\sigma\in\bar{\Y}^*_{2n+1,k}$. We define 
$\bar{\varphi}_x$ as follows.
\begin{enumerate}
\item
 If $a(2i)$ is even then we factorize $\sigma=\tau_3'\tau_2'\tau_1'(2i)\tau_1\tau_2\tau_3$, where $\tau_1$ and $\tau_2'$ is the maximal sequence of consecutive entries greater than $2i$, $\tau_2$ and $\tau_1'$ is the maximal sequence of consecutive entries smaller than $2i$. We define
$$
\bar{\varphi}_{x}(\sigma)=\begin{cases}
\tau_3'\tau_2'\tau_1'\tau_1\tau_2(2i)\tau_3&\textrm{if $l_{\sigma}(x)$ is even};\\
\tau_3'(2i)\tau_2'\tau_1'\tau_1\tau_2\tau_3&\textrm{if $l_{\sigma}(x)$ is odd}.
\end{cases}
$$
\item
If $a(2i-1)$ is even. 
\begin{enumerate}
\item
If $l_{\sigma}(2i-1)=2b$ for some $b\in\N$, then 
$$
\bar{\varphi}_{2i-1}(\sigma)=w\in\bar{\varphi}_{2n+1,k}^*,
$$
where $w$ is constructed as follows.
\\
If $2i-1$ is a valley (resp. peak) of $\sigma$, then insert $2i-1$ in $w_0$ such that at the left side of $2i-1$ in $w_{01}$, the number of valleys minus the numbers of peaks equals to $l_{2i-1}(\sigma)+1$ (here, valleys and peaks are those element smaller than $2i-1$ in $\sigma$). Then we insert one by one the rest peaks and valleys are greater than $2i-1$ in $\sigma$ into $w_{01}$ in increasing order to obtain $w_1$, such that those peaks and valleys say $2j-1$s are still peaks and valleys in $w_1$ and  for every $l_{\sigma}(2j-1)=l_{w_1}(2j-1)$. Last we insert the rest element of $\sigma$ in $w_1$ one by one in increasing order such that $l_w(i)=l_{\sigma}(i)$ for all $[2n+1]\setminus\{2i-1\}$. 
\item
If $l_{\sigma}(2i-1)=2b+1$ for some $b\in\N$, then 
$$
\bar{\varphi}_{2i-1}(\sigma)=w\in\bar{\Y}_{2n+1}^*,
$$
where $w$ is constructed by following almost the process in (a), except when obtain $w_{01}$ by inserting $2i-1$ in $w_0$ such that at the left side of $2i-1$ in $w_{01}$, the number of valleys minus the number of peaks equals to $l_{2i-1}(\sigma)-1$.
\end{enumerate}
\end{enumerate}
Clearly 
the action $\bar{\varphi}_x$ is an involution on $\overbar{\Y}_{2n+1}^*$, besides two actions 
 $\bar{\varphi}_x$ and $\bar{\varphi}_y$ commute for all $x,y$ are VOP pairs. Hence for any subset $X\subset H$ where $H$ is the set of VOP pairs of $\sigma$, we may define the action $\varphi_X$ on $\sigma\in\overbar{\X}_{2n}$ by
\begin{equation}\label{def:action}
\bar{\varphi}_X(\sigma)=\prod_{x\in X}\bar{\varphi} _x(\sigma).
\end{equation}
Following the similar proof of Lemma~\ref{lem3} and Lemma~\ref{lem4}, we also have the following two Lemmas for
$\mathcal{F}$-permutations.
\begin{lem}\label{lem01}
Let $\sigma\in\overline{\Y}_{2n+1}^*$, for any $x$ a VOP pair of $\sigma$, the permutation $\overline{\varphi}_x(\sigma)$ has the same set of valleys and peaks as $\sigma$. Let 
$$
\Delta_{\bar{\varphi}_x}(\sigma)=(\231\bar{\varphi}_x(\sigma),\312\bar{\varphi}_x(\sigma))
-(\231\sigma,\312\sigma).
$$ 
Then we have,
\begin{numcases}{\Delta_{\bar{\varphi}_x}(\sigma) =}
  (-1, 1)& if $l_{\sigma}(x)$  is even\nonumber
   \\
   (1, -1)&if  $l_{\sigma}(x)$ is odd \nonumber
\end{numcases}
\end{lem}
\begin{lem}\label{lem02}
For each $\hat{\sigma}\in\widehat{\Y}_{2n+1}^*$, let $Orb(\hat{\sigma})$ be the orbit of $\hat{\sigma}$ under $\bar{\varphi}_H$. Then
$$
\sum_{\sigma\in Orb(\hat{\sigma})}p^{\res\sigma}q^{\312\sigma}=(p+q)^{2k}p^{\res\sigma-2k}q^{\312\sigma}. 
$$
\end{lem}
Combine Lemma~\ref{lem01} and Lemma~\ref{lem02}, we obtain the following Theorem.
\begin{thm}\label{oddodd:gammafactorization}
We have
$$
\gamma_{n,k}(1,p,q)=(p+q)^{2k}\sum_{\sigma\in\widehat{\Y}^*_{2n+1,k}}p^{\res\sigma-2k}q^{\312\sigma}.
$$
\end{thm}

Let
\begin{subequations}
\begin{align}
\overline{Y}_n(a,p,q,t)&=\sum_{\sigma\in\overline{\Y}^*_{2n+1}}a^{\lema\,\sigma}p^{\res\sigma}q^{\312\sigma}t^{\des\sigma}\\
\widehat{Y}_n(p,q,t)&=\sum_{\sigma\in\widehat{\Y}^*_{2n+1}}p^{\res\sigma}q^{\312\sigma}t^{\des\sigma}.
\end{align}
From Theorems~\ref{oddodd:cf}, \ref{oddodd:gamma} and \ref{oddodd:gammafactorization}  we derive immediately 
the following identities.
\begin{cor}
We have
\begin{align}
\overline{Y}_{n}(a,p,q, t)&=\overline{X}_{n}(a,1, p,q, t),\\
\widehat{Y}_n(p,q,t)&=\widehat{X}_{n}(1,1, p,q, t).
\end{align}
\end{cor}
\end{subequations}
\begin{remark}
The \emph{normalized
$\mathcal{E}$-permutations} and 
\emph{normalized $\mathcal{F}$-permutations} are  two  new models for the normalized median Genocchi numbers. It would be interesting to 
establish  connections of these models with the well-known  \emph{Dellac configurations} 
for normalized Genocchi numbers  
in \cite{HZ99, Bi14}.
\end{remark}

\bibliography{Genocchi-1.bib}{}

\bibliographystyle{plain}
\end{document}